\documentclass[12pt]{amsart}
\usepackage{amsmath}
\usepackage{extarrows}
\usepackage{amsfonts}
\usepackage{amssymb}
\usepackage[all,cmtip]{xy}           

\usepackage{bbm}
\usepackage{bbding}
\usepackage{txfonts}
\usepackage[shortlabels]{enumitem}
\usepackage{ifpdf}
\ifpdf
\usepackage[colorlinks,final,backref=page,hyperindex]{hyperref}
\else
\usepackage[colorlinks,final,backref=page,hyperindex,hypertex]{hyperref}
\fi
\usepackage{tikz-cd}
\usepackage[active]{srcltx}
\usepackage{tikz}
\usepackage{amscd}
\usepackage{bm}
\usepackage{mathrsfs}

\makeatletter

\newtheorem{thm}{Theorem}[section]
\newtheorem{prop}[thm]{Proposition}
\newtheorem{lem}[thm]{Lemma}
\newtheorem{cor}[thm]{Corollary}
\newtheorem{prop-def}{Proposition-Definition}[section]

\newtheorem{defn}[thm]{Definition}

\newtheorem{rmk}[thm]{Remark}
\newtheorem{exam}[thm]{Example}

\newcommand{\nc}{\newcommand}
\nc{\Sh}{{\mathrm{Sh}}}
\nc{\ot}{\otimes}
\nc{\Id}{\mathrm{Id}}
\nc{\id}{\mathrm{Id}}
\nc{\NR}{{\rm NR}}
\nc{\G}{{\rm G}}
\nc{\T}{{\rm T}}
\nc{\Hom}{\mathrm{Hom}}
\nc{\rmC}{ {\mathrm{C}}}
\nc{\rmS}{ {\mathrm{S}}}
\nc{\RB}{{\mathsf{RB}}}
\nc{\ad}{\mathrm{ad}}
\nc{\Lie}{\mathrm{Lie}}
\nc{\RBA}{{\mathrm{RBA}_{(\mu,\lambda)}}}
\nc{\C}{\mathrm{C}}
\nc{\rmH}{\mathrm{H}}
\nc{\rar}{\rightarrow}
\nc{\End}{\mathrm{End}}

\def\bc{\begin{center}}
	\def\ec{\end{center}}

\def\hang{\hangindent\parindent}
\def\textindent#1{\indent\llap{\qquad #1\ \ \enspace}\ignorespaces}
\def\ref{\par\hang\textindent}

\def\o{\otimes}

\def \frakg{ \frak g}
\def\bfk{\mathbf{k}}

\def \frakg{\mathfrak{g}}
\def \frakL{\mathfrak{L}}
\def \frakm{\mathfrak{m}}
\def \fraka{\mathfrak{a}}

\allowdisplaybreaks
\begin{document}
	\title[Extended Rota-Baxter Lie algebras]{Deformations and the  controlling $L_{\infty}$-structure of extended Rota-Baxter Lie  algebras}
	
	\author{Jian Yang}
	\address{Jian Yang\\
		School of Mathematical Sciences, 
		East China Normal University,
		Shanghai 200241,
		China}

	\email{y.j0@qq.com }
	
	\date{\today}
	
	\begin{abstract} In this paper, we introduce the controlling $L_{\infty}[1]$-algebra of extended Rota-Baxter Lie algebras. As applications, we obtain  the cohomology of  extended Rota-Baxter Lie algebras. We also study  abelian extensions and formal deformations by the lower cohomology group. Lastly, we get the notion of homotopy extended Rota-Baxter Lie algebras.  
		
		Specifically, we present the controlling $L_{\infty}[1]$-algebra  of modified Rota-Baxter Lie algebras and of Rota-Baxter Lie algebras with weight.
	\end{abstract}
	
	
	
	\maketitle
	
	\tableofcontents
	
	\allowdisplaybreaks

\section{Introduction}

In this paper, we study  extended Rota-Baxter Lie algebras,  a generalization of Rota-Baxter Lie algebras with weight and modified Rota-Baxter Lie algebras. 

	\subsection{Extended Rota-Baxter Lie algebras}\

Baxter introduced Rota-Baxter algebras  in \cite{Bax60},  interestingly, this work originated from research in probability theory. Then Rota \cite{Rot69}, Cartier \cite{Car72} and other researchers subsequently entered this field.   Semenov-Tian-Shansky  independently discovered that the Rota-Baxter operator on  Lie algebras  is a solution of the classical Yang-Baxter equation \cite{Sem83}. Similarly, for the case of associative algebras, it was given by Aguiar \cite{Agu00} and Bai \cite{Bai07}. Subsequently, Guo and Keigher \cite{GK00a, GK00b} realized free commutative Rota-Baxter algebras. In  \cite{Sem83}, the authors also introduced a modified version of the classical Yang-Baxter
equation  whose solutions are just the modified Rota-Baxter operators, which they call  modified r-matrices.

Connes and Kreimer  \cite{CK00} established a significant connection between Rota-Baxter operators and mathematical physics  in the renormalization of quantum field theory.
Rota-Baxter operators yield a wealth of applications and linkages across different mathematical domains such as combinatorics \cite{Rot95}, multiple zeta values in number theory \cite{GZ08}, operad theory \cite{Agu01, BBGN13}, and Hopf algebras \cite{CK00}. The modified classical Yang-Baxter
equation has important applications in the study of Lax equations, affine geometry on Lie groups, factorization problems
in Lie algebras and compatible Poisson structures \cite{BGN10,B90,L99,S15}.

Due to the importance of the Rota-Baxter operator and the
modified Rota-Baxter operator on Lie algebras, we study their generalization: the extended Rota-Baxter Lie algebra.

Fix $\mu,\lambda\in \bfk$ where $\bfk$ is a field. An extended Rota-Baxter operator of weight $(\mu,\lambda)$ on a Lie algebra $(A,[-,-])$   is a linear operator $T:A\to A$ satisfying the identity:
\begin{equation} \label{eq:diffwt}
	[T(u),T(v)]=T([T(u),v]+[u,T(v)]+\mu [u,v])+\lambda [u,v], \quad u, v\in A.
\end{equation}
A Lie algebra  endowed with an extended Rota-Baxter operator of weight $(\mu,\lambda)$ is called an extended Rota-Baxter Lie algebra of weight $(\mu,\lambda)$.

\subsection{Deformations and homotopy theory}\

A central philosophy in deformation theory, inspired by the foundational work of Gerstenhaber, Nijenhuis, Richardson, Deligne and others, is that the deformation of a given mathematical structure can be described by a differential graded (dg) Lie algebra or, more generally, an $L_\infty$-algebra. Consequently, an essential problem in deformation theory is the explicit construction of the dg Lie algebra or $L_\infty$-algebra that governs the deformations.

Another significant challenge in the study of algebraic structures is understanding their homotopy analogs, such as $A_\infty$-algebras for associative algebras and $L_\infty$-algebras for Lie algebras.

The deformation theory  and homotopy theory of  Rota-Baxter  algebras had been absent for a long time despite   the   importance of Rota-Baxter   algebras.  Recently there are some breakthroughs in this direction.

Fortunately, there exists a powerful technique of  derived brackets which was invented by   Voronov   \cite{Vor05}. It has been successfully applied to obtain the controlling algebra of relative Rota-Baxter Lie algebra of weight zero \cite{LST,LST2} and the controlling algebra of relative Rota-Baxter algebra of weight zero \cite{DM20}. However this technique cannot apply directly to the general case of extended Rota-Baxter Lie algebras. We use a generalised version to study Rota-Baxter Lie algebras of arbitrary weight in this paper as a supplement. This provides further strong evidence that the controlling algebra of extended Rota-Baxter Lie algebras is a good one. 

The  case of  nonzero weight is more difficult. Wang and Zhou \cite{WZ24} gave an explicit construction of the minimal model of the Rota-Baxter operad and developed the deformation and homotopy theory of Rota-Baxter algebras of arbitrary weight. Recently, Chen, Wang and Zhou established those results for Rota-Baxter Lie algebras \cite{CWZ}.  Specifcally, the controlling $L_{\infty}[1]$-algebra of Rota-Baxter algebras with weight coincides with ours.

For modified Rota-Baxter algebras, Das \cite{Das22} gave the cohomology theory and applications.

\subsection{Layout of the paper}\

We recall some basic notions and  introduce some facts about  extended Rota-Baxter Lie algebras and their representations in Section~\ref{sec 2}. In particular, there is a key descending property for extended Rota-Baxter representations in Proposition~\ref{descending property}.

In Section \ref{sec 3}, we recall the curved dg Lie algebra and the MC characterization of Lie algebras. We also give the MC characterization and  cohomology of the extended Rota-Baxter operator on Lie algebras in \ref {subsec 3.2}.

In the next section \ref{sec 4}, we recall the $L_\infty$-algebra, since  the controlling algebra of extended Rota-Baxter Lie algebras is an $L_{\infty}$-algebra .

The following Sections \ref{sec 5}, \ref{sec 6}  are the key sections of our paper: we construct the $L_{\infty}[1]$-algebra of extended Rota-Baxter Lie algebras in Theorem \ref{thm: Linfinity for extend Lie} and cohomology of extended Rota-Baxter Lie algebras with arbitrary coefficients in Proposition-Definition \ref {RBLie homology on V}. And we consider those important cases: modified Rota-Baxter Lie algebras and of Rota-Baxter Lie algebras with weight.

In Sections  \ref{sec 7}, \ref{sec 8} and \ref{sec 9}, we study formal deformations, abelian extensions  and homotopy extended Rota-Baxter Lie algebras, as one would expect of a good cohomology theory and of a good controlling algebra of extended Rota-Baxter Lie algebras.

	\bigskip

	\section{Extended  Rota-Baxter Lie algebras}\label{sec 2}
	
	Throughout this paper,	$\bfk$ is a field  with $\mathrm{char}(\bfk)=0$.
	Let $V=\oplus_{n\in \mathbb{Z}} V^n$ be a graded vector space. We define the shift map $s$  by $(sV)^n:=V^{n+1}$.

	Now we introduce extended Rota-Baxter Lie algebras in detail.

	\begin{defn}
		\textbf{An extended Rota-Baxter Lie algebra of weight $(\mu,\lambda)$} for $\mu,\lambda \in \bfk$ is a triple $(\frakg,\pi=[-,-],T)$ where $(\frakg,\pi)$ is a Lie algebra and  $T:\frakg\to \frakg$ is a linear map such that $[T(a),T(b)]=T([T(a),b]+[a,T(b)]+\mu [a,b]) +\lambda [a,b]$ for $a,b \in \frakg$.
		
		A morphism of extended Rota-Baxter Lie algebras  between $(\frakg,\pi,T)$ and $(\frakg',\pi',T')$ of the same weight is a morphism of Lie algebras $f: \frakg \to \frakg'$ such that $f\circ T= T'\circ f$.
	\end{defn}

	\begin{defn}\label{Def: Rota-Baxter rep}
		Let $(\frakg,\pi,T)$ be an extended Rota-Baxter Lie algebra of weight $(\mu,\lambda)$ and $(M,\rho)$ be a representation
		of the Lie algebra $\frakg$. We say that $M$ is an extended Rota-Baxter representation if  $M$ is endowed with a
		linear map $T_M: M\rightarrow M$ such that the following
		equations
		\begin{eqnarray}\rho(T(a))T_M(m)&=&T_M\big(\rho(a)T_M(m)+\rho(T(a))m+\mu \rho(a)m\big)+\lambda \rho(a)m,
		\end{eqnarray}
		hold for any $a\in \frakg$ and $m\in M$.
	\end{defn}
	
	\begin{rmk}
		Set $[a,m]:= \rho(a)m$.
	\end{rmk}
	
	\begin{exam}
		A Rota-Baxter Lie algebra of weight $\lambda$ is precisely an extended Rota-Baxter Lie algebra of weight $(\lambda,0)$.
		
		A modified Rota-Baxter Lie algebra of weight $\lambda$ is just an extended Rota-Baxter Lie algebra of weight $(0,\lambda)$.
	\end{exam}

	\begin{exam}
		$(\frakg,T)$ itself is an extended Rota-Baxter representation of $(\frakg,\pi,T)$, it is called the regular extended Rota-Baxter representation.
	\end{exam}

	The following is easy to check:
	\begin{prop}  \label{Prop: trivial extension of Rota-Baxter rep}
		Let $(\frakg,\pi,T)$ be an extended Rota-Baxter Lie algebra of weight $(\mu,\lambda)$ and let  $M$ be an extended Rota-Baxter representation of $(\frakg,\pi,T)$. Then $(\frakg\oplus M,T\oplus T_M)$ is an extended Rota-Baxter Lie algebra with the bracket defined by \begin{eqnarray}[(a,m),(b,n)]=([a, b], [a,n]-[b,m]).\end{eqnarray}
		
		  Denote this new extended Rota-Baxter Lie algebra by $\frakg\ltimes  M$. It is called the \textbf{semi-direct product} (or trivial extension) of $\frakg$ by $M$.
	\end{prop}

	We have the  following interesting observation:
	\begin{prop}\label{Prop: new RB algebra}
		Let $(\frakg,\pi,T)$ be an extended Rota-Baxter Lie algebra of weight $(\mu,\lambda)$. Define a new bracket operation as follows:
		\begin{eqnarray}[a , b]_\star:=[a, T(b)]+[T(a), b]+\mu [a, b]\end{eqnarray}
		for any $a,b\in \frakg$. Then the triple  $(\frakg,[-,-]_\star ,T)$ also forms an extended Rota-Baxter Lie algebra of weight $(\mu,\lambda)$  and we denote it by $\frakg_\star $.
	\end{prop}
	\begin{proof}
		We only show that $T$ is an extended Rota-Baxter operator on the Lie algebra $\frakg_\star$.
		
		For any $a,b\in \frakg$,
		\begin{align*}
			&[Ta, Tb]_\star\\
			=& [T(a),T^2(b)]+[T^2(a),T(b)]+\mu [T(a),T(b)]\\
			=&T\big([a,T^2(b)]+[T(a),T(b)]+\mu[a,T(b)]\big)+\lambda[a,T(b)]\\
			&+T\big([T(a),T(b)]+[T^2(a),b]+\mu[T(a),b]\big)+\lambda[T(a),b]\\
			&+\mu T\big([T(a),b]+[a,T(b)]+\mu[a,b]\big)+\lambda\mu[a,b]\\
		\end{align*}
	On the other hand,
	 \begin{align*}
	 	&T\big([T(a), b]_\star+[a,T(b)]_\star+\mu[a,b]_\star\big)+\lambda [a,b]_\star\\
	 	=&T\big([T(a),T(b)]+[T^2(a),b]+\mu [T(a),b]\big)\\
	 	&+T\big([a,T^2(b)]+[T(a),T(b)]+\mu[a,T(b)]\big)\\
	 	&+\mu T\big([T(a),b]+[a,T(b)]+\mu[a,b]\big)\\
	 	&+\lambda \big([T(a),b]+[a,T(b)]+\mu[a,b]\big)\\
	 	=&[Ta, Tb]_\star
	 \end{align*}
	\end{proof}

	One can also construct the new extended Rota-Baxter representation from old ones. 
	\begin{prop}\label{descending property}
		Let $(\frakg,\pi,T)$ be an extended Rota-Baxter Lie algebra of weight $(\mu,\lambda)$ and $(M,T_M)$ be an extended Rota-Baxter representation over it. Define a new action $[-,-]_t$ of $\frakg$ on $M$ as follows: for any $a\in \frakg,m\in M$,
		\begin{eqnarray}
			[a, m]_t:&=& [T(a),m]-T_M([a,m]).
		\end{eqnarray}
		Then these actions make $M$ into an extended Rota-Baxter representation over $\frakg_\star$ and denote this new   extended Rota-Baxter representation by $M_t$.
	\end{prop}
	  
	  \begin{proof}
	 	Firstly, notice that $M_t$ is a representation over the Lie algebra $\frakg_\star$ and we leave it as an exercise.
	 	
	 	Then we show that $M_t$ is an extended Rota-Baxter representation over $\frakg_\star$. That is, for any $a\in \frakg_\star$ and $m\in M_t$,  
	 	
	 	$$\begin{array}{rcl}
	 		[T(a), T_M(m)]_t&=&[T^2(a),T_M(m)]-T_M[T(a),T_M(m)]\\
	 		&=&  T_M ([T(a),T_M(m)] + [T^2(a),m] +\mu  [T(a),m] )+\lambda [T(a),m]-T_M[T(a),T_M(m)]  \\
	 		&=& T_M\big(   [T^2(a),m] +\mu  [T(a),m] \big)+\lambda [T(a),m] 
	 	\end{array}$$
	 	and
	 	$$\begin{array}{rl}
	 		& T_M\big([a, T_M(m)]_t+[T(a), m]_t+\mu [a, m]_t\big)+\lambda [a, m]_t\\
	 		=&T_M\big( [T(a),T_M(m)]-T_M[a,T_M(m)]+ [T^2(a),m]-T_M[T(a),m]+\mu  [T(a),m]-\mu T_M[a,m]\big)\\
	 		&+\lambda ([T(a), m]-T_M[a,m])\\
	 		=& T_M\big( T_M\big([a,T_M(m)]+[T(a),m]+\mu [a,m]\big)+\lambda [a,m] -T_M[a,T_M(m)]+ [T^2(a),m]-T_M[T(a),m]\\
	 		&+\mu  [T(a),m]-\mu T_M[a,m]\big)+\lambda ([T(a), m]-T_M[a,m])\\
	 		=& T_M\big(   [T^2(a),m] +\mu [T(a),m] \big)+\lambda [T(a), m]\\
	 		=& [T(a), T_M(m)]_t.
	 	\end{array}$$
	 \end{proof}

	\bigskip

	\section{The MC characterization and the cohomology of Lie algebras and of extended Rota-Baxter operators}\ \label{sec 3}

	 As ingredients of extended Rota-Baxter Lie algebras, we now provide a further characterization of Lie algebras and extended Rota-Baxter operators.
	 
	\subsection{Curved dg Lie algebras}\ \label{subsec 3.1}
	
	The controlling algebra of extended Rota-Baxter operators  is not a dg Lie algebra but a curved dg Lie algebra. So it is necessary to introduce it.
	
	\begin{defn}\label{curved dg Lie}
		\textbf{A curved dg Lie algebra} is a quadruple $(\frakg,d_0,d_1,d_2)$ where $\frakg$ is a graded space and three graded linear operators $d_0:\bfk \rightarrow \frakg,d_1:L\rightarrow \frakg$ and $d_2:\frakg\otimes \frakg\rightarrow \frakg$  satisfy the following conditions :\\
		$(0) \quad d_2(x,y)=(-1)^{|x||y|} d_2(y,x)$,\\
		$(1) \quad d_1\circ d_0=0$,\\
		$(2) \quad d_1\circ d_1(x) +d_2(d_0(1), x)=0$,\\
		$(3)\quad d_1\circ d_2+d_2\circ (d_1\otimes Id+ Id\otimes d_1)=0$,\\
		$(4) \quad d_2(d_2(x,y),z)+(-1)^{|x||y|+|x||z|}d_2(d_2(y,z),x)+(-1)^{|y||z|}d_2(d_2(x,z),y)=0$ (Jacobi identity).\\
		For homogeneous elements $x,y,z\in \frakg$.
	\end{defn}
	
	\begin{rmk}
		If $d_0=0$, it's just a dg Lie algebra. Moreover, if $d_0=0,d_1=0$, it's just a graded Lie algebra. Note that it  is slightly different from the definition of ordinary graded Lie algebras. 
		
			\textbf{An ordinary graded  Lie algebra} is a pair $(\frakg,[-,-])$ where $\frakg$ is a graded space with the bracket $[-,-]:\frakg\otimes \frakg\rightarrow \frakg$ that satisfies the following conditions :\\
		$(1) \quad [x,y]=-(-1)^{|x||y|} [y,x]$,\\
		$(2) \quad [[x,y],z]+(-1)^{|x||y|+|x||z|}[[y,z],x]-(-1)^{|y||z|}[[x,z],y]=0$ (Jacobi identity),\\
		for homogeneous elements $x,y,z\in \frakg$.
		
		Given an ordinary graded  Lie algebra  $(s\frakg,[-,-])$, then $(\frakg,\pi)$ is a graded  Lie algebra where $\pi(sx,sy)=(-1)^{|x|}s[x,y]$ for  homogeneous elements $x,y\in \frakg$.
	\end{rmk}
	
	\begin{defn}
		Given a curved dg Lie algebra $(\frakg,d_0,d_1,d_2)$. An element $\alpha\in \frakg^0$ is called a Maurer-Cartan element if and only if it satisfies the Maurer-Cartan equation:\\
		$$d_0(1)+d_1(\alpha)+\frac{1}{2}d_2(\alpha,\alpha)=0.$$
	\end{defn}
	
	\begin{prop}[Twisting procedure] \label{Prop: twist curved dg Lie}
		Let  $\alpha$ be a Maurer-Cartan element of the curved dg Lie algebra $\frakg$. The twisted dg Lie algebra on $\frakg$ is given by $d_n ^{\alpha}: \frakg^{\ot n}\rightarrow \frakg$ for $n=1,2$ which is defined as follows$\colon$
		\begin{eqnarray*}\label{Eq: twisted curved dg Lie algebra}  \quad d^\alpha_1(x)=d_{1} (x)+d_2(\alpha,x),\quad d^\alpha_2(x,y)=d_2(x,y),\ \forall x,y\in \frakg.\end{eqnarray*}
		Denote $\mathcal{MC}(\frakg):=\{\mbox{Maurer-Cartan elements of}~ \frakg\}$.
	\end{prop}
	
	\medskip
	
	\subsection{The MC characterization and the cohomology of Lie algebras}\ \label{subsec 3.2}
	
	We fix an integer $n\geq 1$.
	For $0\leq i_1, \dots, i_r\leq n$ with $i_1+\cdots+i_r=n$,  $\Sh(i_1, i_2,\dots,i_r)$ is the   set of $(i_1,\dots, i_r)$-shuffles, i.e., the permutation $\sigma\in S_n$ such that
	$$\sigma(1)<\sigma(2)<\dots<\sigma(i_1),  \ \sigma(i_1+1)< \dots<\sigma(i_1+i_2),\ \dots,\
	\sigma(i_1+\cdots+i_{r-1}+1)< \cdots<\sigma(n).$$
	
	\begin{defn}(\cite{NR67})
		Let $\frakg$ be a vector space. Consider the graded vector space $ \mathrm{Hom}(\bigwedge (\frakg),\frakg):=\bigoplus_{n=-1}^{\infty}\mathrm{Hom}(\wedge^{n+1}\frakg,\frakg)$ and we say the degree  of $f$ is $n$ if $f\in\mathrm{Hom}(\wedge^{n+1}\frakg,\frakg)$. 
     	Define the  \textbf{Nijenhuis-Richardson bracket} $[\cdot,\cdot]_{\bf{NR}}$ on $\mathrm{Hom}(\bigwedge (\frakg),\frakg)$  as follows:
		for any $P\in \mathrm{Hom}(\wedge^{p+1}\frakg,\frakg)$ and $Q\in \mathrm{Hom}(\wedge^{q+1}\frakg,\frakg)$,
		$$[P,Q]_{\bf{NR}}=P\bar{\circ} Q-(-1)^{pq}Q\bar{\circ} P,$$
		where $P\bar{\circ} Q$ is defined by
		\begin{align*}
			(P \bar{\circ} Q)(x_{1}, \ldots, x_{p+q+1})=\sum_{\sigma \in \Sh(q+1, p)}sgn(\sigma) P(Q(x_{\sigma(1)}, \ldots, x_{\sigma(q+1)}), x_{\sigma(q+2)}, \ldots,  x_{\sigma(p+q+1)}).
		\end{align*}
	\end{defn}
	
	\begin{prop}\label{gla for Lie}\
		
		\begin{enumerate}
			\item 	Let $\frakg$ be a vector space. Then $(s\mathrm{Hom}(\bigwedge (\frakg),\frakg),[\cdot,\cdot])$ is a graded Lie algebra(aka. gla) where we define $[sx,sy]:=(-1)^{|x|}s[x,y]_{\bf{NR}}$ for any $x\in \mathrm{Hom}(\wedge^{p+1}\frakg,\frakg)$ and $y\in \mathrm{Hom}(\wedge^{q+1}\frakg,\frakg)$.
			\item An element  $s\pi \in s\mathrm{Hom}(\wedge^{2}\frakg,\frakg)$  is a Maurer-Cartan element   if and only if $\pi$ is a Lie bracket on $\frakg$.
		\end{enumerate}
	\end{prop}

   \begin{rmk}
   	$(\mathrm{Hom}(\bigwedge (\frakg),\frakg),[\cdot,\cdot]_{\bf{NR}})$ is an ordinary graded Lie algebra by \cite{NR67}. It follows that the above proposition holds.
   \end{rmk}
	
	Let $(\frakg,\pi)$ be a Lie algebra. Then  $s\pi$  is a Maurer-Cartan element of $(s\mathrm{Hom}(\bigwedge (\frakg),\frakg),[\cdot,\cdot])$, hence twisting by $s\pi$, we have a new dg Lie algebra. 
	
	\begin{defn}\label{cohomology of algebra}
		Let $(\frakg,\pi)$ be a Lie algebra. \textbf{The cochain complex of $\frakg$} is defined to be $(\rmC^*_{\Lie}(\frakg):=\mathrm{Hom}(\bigwedge (\frakg),\frakg),\partial_{\Lie}:=[\pi,-]_{\bf{NR}})$ induced by the new dg Lie algebra above. The corresponding cohomology is called \textbf{the cohomology of $\frakg$}.  We could see that they are just the Chevalley-Eilenberg cochain complex and Chevalley-Eilenberg cohomology of $\frakg$. More precisely,\\
		
		the coboundary operator $$\partial_{\Lie}^n: \mathrm{Hom}(\wedge^n \frakg,\frakg)\longrightarrow  \mathrm{Hom}(\wedge^{n+1} \frakg,\frakg), n\geq 0$$ is given by
		\[\begin{split}
			\partial_{\Lie}^n (f)(x_1,\dots,x_{n+1})&=\sum_{i=1}^{n+1}(-1)^{i+n}[x_i, f(x_1,\dots,\hat{x}_i, \dots, x_{n+1})]\\
			&+\sum_{1\leq i<j\leq n+1} (-1)^{i+j+n+1}f([x_i, x_j], x_1,\dots,\hat{x}_i, \dots,\hat{x}_j, \dots,x_{n+1}),
		\end{split}\]
		for all $f\in \mathrm{Hom}(\wedge^n \frakg,\frakg),~x_1,\dots, x_{n+1}\in \frakg$, where $\hat{x_i}$ means deleting this element.
	\end{defn}

	\medskip
	
	\subsection{The MC characterization and the cohomology of extended Rota-Baxter operators}\
	
	Let $(\frakg,\pi)$ be a Lie algebra.

	  Define three graded linear operators as follows,  for $f\in\mathrm{Hom}(\wedge^{n+1}\frakg,\frakg), g\in \mathrm{Hom}(\wedge^{m+1}\frakg,\frakg)$.
	
	$d_0:\bfk \rightarrow \mathrm{Hom}(\bigwedge (\frakg),\frakg)$ is given by $d_0(1)=-\lambda \pi$,\\
	$d_1=d:\mathrm{Hom}(\bigwedge (\frakg),\frakg)\rightarrow \mathrm{Hom}(\bigwedge (\frakg),\frakg)$ is given by $d_1(f)=(-1)^{|f|+1}\mu f\bar{\circ} \pi$,\\
	and $d_2=[-,-]:\mathrm{Hom}(\bigwedge (\frakg),\frakg)\otimes \mathrm{Hom}(\bigwedge (\frakg),\frakg)\rightarrow \mathrm{Hom}(\bigwedge (\frakg),\frakg)$ is given by \\
	$d_2(f,g)(x_1,\dots,x_{m+n+2})=\sum_{\sigma\in \Sh(n+1,m+1)}sgn(\sigma)(-1)^{(n+1)m}\pi(f(x_{\sigma(1)},\dots,x_{\sigma(n+1)}),g(x_{\sigma(n+2)},\dots,x_{\sigma(n+m+2)}))$\\
	$-(-1)^{(n+1)m}\sum_{\sigma\in \Sh(n+1,1,m)}sgn(\sigma)g(\pi(f(x_{\sigma(1)},\dots,x_{\sigma(n+1)}),x_{\sigma(n+2)}),x_{\sigma(n+3)},\dots,x_{\sigma(n+m+2)})$\\
	$-(-1)^{n}\sum_{\sigma\in \Sh(m+1,1,n)}sgn(\sigma)f(\pi(g(x_{\sigma(1)},\dots,x_{\sigma(m+1)}),x_{\sigma(m+2)}),x_{\sigma(m+3)},\dots,x_{\sigma(n+m+2)})$.

		\begin{thm}
		With the above notation, then $(\mathrm{Hom}(\bigwedge (\frakg),\frakg),d_0,d, [-,-])$ forms a curved dg Lie algebra.
	\end{thm}

	\begin{rmk}
		The above theorem can be proved by direct computation. Here we will give an alternative proof in Corollary \ref{curved dg Lie of operator} by using
		Theorem \ref{thm: Linfinity for extend Lie}.
	\end{rmk}

	\begin{thm}\label{MC elements Lie}
		With the above notation.  Let $T\in\mathrm{Hom}(\frakg,\frakg)$, then
		$T\in \mathcal{MC}(\mathrm{Hom}(\bigwedge (\frakg),\frakg))$ if and only if $T$ is  an extended Rota-Baxter operator of weight  $(\mu,\lambda)$ on $\frakg$.
	\end{thm}
	\begin{proof}
		\begin{equation*}
			\begin{aligned}
				0 &= 	d_0(1)+d(T)+\frac{1}{2}[T,T]  \\
				&=  -\lambda \pi-\mu T\circ \pi+ \frac{1}{2} [T,T].
			\end{aligned}
		\end{equation*}
		This equation of $T$ coincides with the one obtained in the proof of Theorem \ref{Thm: MC elements in ex Linifnity}.
	\end{proof}
	
	By Theorem~\ref{MC elements Lie}, an extended Rota-Baxter operator $T$ is a Maurer-Cartan element in the curved dg Lie algebra $\mathrm{Hom}(\bigwedge (\frakg),\frakg)$. Twisting $\mathrm{Hom}(\bigwedge (\frakg),\frakg)$ by  $T$ yields a new differential.   
	
	\begin{prop-def}\label{cohomology of operator}
		Consider a Lie algebra $\frakg$ and $T$ is an extended Rota-Baxter operator on $\frakg$. \textbf{The cochain complex of the extended Rota-Baxter operator  $T$} is defined to be $(\rmC^*_{\RB}(\frakg):=\mathrm{Hom}(\bigwedge (\frakg),\frakg),\partial_\RB:=d_{1}^{T})$. The corresponding  cohomology group is called \textbf{the cohomology of the extended Rota-Baxter operator  $T$}. One may observe that this complex (and cohomology group) coincides with the Chevalley-Eilenberg cochain complex (and cohomology group)  of $\frakg_\star$ with coefficients in $\frakg_t$(See Proposition  \ref{Prop: new RB algebra} and \ref{descending property}).
	\end{prop-def}

	\begin{proof}
		For   $x_1,\dots,x_{n}\in \frakg$ and $f\in \Hom(\wedge^{n-1} \frakg,\frakg)$, we have
		\begin{eqnarray*}
			&&(d_1(f)+d_2(T, f))(x_1, \dots, x_{n})\\
			&=&\sum_{1\leq i<j\leq n}(-1)^{i+j+n}\mu f([x_i, x_j]_\mathfrak{g}, x_1,\dots,\hat{x}_i, \dots,\hat{x}_j, \dots,x_n)+\sum_{i=1}^{n}(-1)^{i+n-1}[T(x_i), f(x_1,\dots,\hat{x}_i, \dots, x_n)]_\mathfrak{g}\\
			&&- \sum_{i=1}^{n}(-1)^{i+n-1}T([x_i, f(x_1,\dots,\hat{x}_i, \dots, x_n)]_\mathfrak{g})\\
			&&+\sum_{1\leq i<j\leq n}(-1)^{i+j+n}f([T(x_i), x_j]_\mathfrak{g} + [x_i, T(x_j)]_\mathfrak{g}, x_1,\dots,\hat{x}_i, \dots,\hat{x}_j, \dots,x_n)\\
			&=&\sum_{i=1}^{n}(-1)^{i+n-1}[T(x_i), f(x_1,\dots,\hat{x}_i, \dots, x_n)]_\mathfrak{g}- \sum_{i=1}^{n}(-1)^{i+n-1}T([x_i, f(x_1,\dots,\hat{x}_i, \dots, x_n)]_\mathfrak{g})\\
			&&+\sum_{1\leq i<j\leq n}(-1)^{i+j+n}f([T(x_i), x_j]_\mathfrak{g} + [x_i, T(x_j)]_\mathfrak{g}+\mu[x_i, x_j]_\mathfrak{g}, x_1,\dots,\hat{x}_i, \dots,\hat{x}_j, \dots,x_n)\\
		\end{eqnarray*}
		induces a new differential on $\mathrm{Hom}(\bigwedge (\frakg),\frakg)$ which differs from the usual Chevalley-Eilenberg cochain complex. And  they are just the Chevalley-Eilenberg cochain complex and cohomology of $\frakg_\star$ with coefficient in $\frakg_t$. 
	\end{proof}

	\bigskip

	\section{$L_\infty[1]$-algebras}\ \label{Subsect: Linfinity algebras}\label{sec 4}
	
	In this section, we  recall some preliminaries on $L_\infty[1]$-algebras. It is well known that $L_\infty[1]$-algebras are equivalent to $L_\infty$-algebras. Consequently, all results parallel to those for $L_\infty$-algebras hold. For $L_\infty$-algebras, see \cite{Get09} for more details.
	
	 Let $V=\oplus_{n\in \mathbb{Z}} V^n$ be a graded vector space. Recall that   the graded symmetric algebra $\rmS(V)$ of $V$ is defined to be the quotient of the tensor algebra $\T(V)$ by   the two-sided ideal $I$   generated by
	$x\ot y -(-1)^{|x||y|}y\ot x$ for all homogeneous elements $x, y\in V$. For $x_1\ot\cdots\ot x_n\in V^{\ot n}\subseteq \T(V)$, write $ x_1\odot x_2\odot\dots\odot x_n$ for its image in $\rmS(V)$.
	For homogeneous elements $x_1,\dots,x_n \in V$ and $\sigma\in \rmS_n$ which is  the symmetric group in $n$ variables, the Koszul sign $\varepsilon(\sigma):=\varepsilon(\sigma;  x_1,\dots, x_n)$ is defined by
	$$ x_1\odot x_2\odot\dots\odot x_n=\varepsilon(\sigma)x_{\sigma(1)}\odot x_{\sigma(2)}\odot\dots\odot x_{\sigma(n)}\in \rmS(V).$$
	Denote by $\rmS^n(V)$ the image of $V^{\otimes n}$ in $\rmS(V)$.

	\begin{defn}\label{Def:L[1]-infty}
		An \textbf{$L_\infty[1]$-algebra} is a graded vector space  $\frakg=\bigoplus\limits_{i\in\mathbb{Z}}\frakg^i$   equipped with   a family of graded linear maps $l_n:\frakg^{\ot n}\rightarrow \frakg, n\geq 1$ of degree $1$ satisfying  the following equations:
		for arbitrary  $n\geq 1$,  $ \sigma\in \rm S_n$ and $x_1,\dots, x_n\in \frakg$,
		\begin{enumerate}
			\item[(i)](graded symmetry)
			\begin{equation*} \label{graded sym}
				l_n(x_{\sigma(1)},\dots,x_{\sigma(n)})=\varepsilon(\sigma)l_n(x_1,\dots,x_n),
			\end{equation*}

			\item[(ii)](generalised Jacobi identity)
			\begin{equation*}\label{graded Jacobi}
				\sum_{i=1}^n\sum_{\sigma\in \Sh(i,n-i)}\varepsilon(\sigma)l_{n-i+1}(l_i(x_{\sigma(1)},\dots,x_{\sigma(i)}),x_{\sigma(i+1)},\dots,x_{\sigma(n)})=0.
			\end{equation*}
			
		\end{enumerate}
		
	\end{defn}

	\begin{rmk} \label{Rem: L[1]-infinity for small n}   Let us consider the generalised Jacobi identity for   $n\leq 3$ with the assumption of  generalised  symmetry.
		
		\begin{enumerate}
			\item[(i)]  For $n=1$,  then $l_1 \circ l_1 =0$, that is,  $l_1 $ is a differential.

			\item[(ii)] For $n=2$, then $l_1\circ l_2 +l_2 \circ (l_1 \ot\Id+\Id\ot l_1)=0$, that is , $l_1$ is a derivation with respect to $l_2$.

			\item[(iii)] For $n=3$ and arbitrary homogeneous elements $x_1, x_2, x_3\in L$, we have
			$$\begin{array}{ll} &l_2 (l_2 (x_1\ot x_2)\ot x_3)+(-1)^{|x_1|(|x_2|+|x_3|)} l_2 (l_2 (x_2\ot x_3)\ot x_1)+
				(-1)^{|x_3|(|x_1|+|x_2|)} l_2 (l_2 (x_3\ot x_1)\ot x_2)
				\\
				=&-\Big(l_1 (l_3 (x_1\ot x_2\ot x_3))+ l_3 (l_1  (x_1)\ot x_2\ot x_3 )+(-1)^{|x_1|} l_3 (x_1\ot l_1  (x_2)\ot x_3 )+\\
				&(-1)^{|x_1|+|x_2|} l_3 (x_1\ot x_2\ot l_1  (x_3) )\Big),\end{array}$$
			that is, $l_2$ satisfies the   Jacobi identity up to homotopy.
		\end{enumerate}

	\end{rmk}

	\begin{defn}
		A \textbf{Maurer-Cartan element} of an $L_\infty[1]$-algebra $(\frakg,\{l_n \}_{n\geq1})$ is  an element $\alpha\in \frakg^{0}$   satisfying the Maurer-Cartan equation:
		\begin{eqnarray*}\label{Eq: mc-equation[1]}\sum_{n=1}^\infty\frac{1}{n!} l_n (\alpha^{\ot n})=0,\end{eqnarray*}
		whenever this infinite sum exists.  Denote $\mathcal{MC}(\frakg):=\{\mbox{Maurer-Cartan elements of}~ \frakg\}$.
	\end{defn}

	\begin{prop}[Twisting procedure] \label{Prop: twist-L-infty[1]}
		Let  $\alpha$ be a Maurer-Cartan element of $L_\infty[1]$-algebra $\frakg$.	The twisted $L_\infty[1]$-algebra  is given by $l_n ^{\alpha}: \frakg^{\ot n}\rightarrow \frakg$ which is defined as follows$\colon$
		\begin{eqnarray*}\label{Eq: twisted L[1] infinity algebra} l^\alpha_n(x_1\ot \cdots\ot x_n)=\sum_{i=0}^\infty\frac{1}{i!}l_{n+i} (\alpha^{\ot i}\ot x_1\ot \cdots\ot x_n),\ \forall x_1, \dots, x_n\in \frakg,\end{eqnarray*}
		whenever these infinite sums exist.
	\end{prop}
	
	\bigskip
	
	\section{The L$_{\infty}$-structure  for extended Rota-Baxter Lie algebras with weight}\label{sec 5}
	
	  From now on, we will always assume that the base field $\bfk$  contains a square root $\sqrt{\mu^2-4\lambda}$ unless otherwise specified.

	  Now we will present the controlling algebra of extended Rota-Baxter Lie algebras.
	\subsection{The $L_{\infty}$-structure for extended Rota-Baxter Lie algebras with weight}\ \label{subsec 5.1}
	
	Before constructing the $L_{\infty}$-algebra, we give some facts about the field $\bfk$.
	\begin{lem} \label{key lemma}
		For $\mu,\lambda \in \bfk$ and $m,n \in \mathbb{Z}$, set $t=\sqrt{\mu^2-4\lambda}$. We have the identity whenever the following  elements exist:
		\begin{equation*}
			\begin{aligned}
				&\frac{\lambda}{t}((\frac{\mu +t}{2})^{m}-(\frac{\mu -t}{2})^{m})\frac{\lambda}{t}((\frac{\mu +t}{2})^{n}-(\frac{\mu -t}{2})^{n})+\frac{\lambda}{t}((\frac{\mu +t}{2})^{n+m+1}-(\frac{\mu -t}{2})^{n+m+1})\\
				&=\frac{\lambda}{t}((\frac{\mu +t}{2})^{m+1}-(\frac{\mu -t}{2})^{m+1})\frac{1}{t}((\frac{\mu +t}{2})^{n+1}-(\frac{\mu -t}{2})^{n+1})	
			\end{aligned}
		\end{equation*}
	\end{lem}

	\begin{rmk}
		We will denote $A_{n}:=\frac{\lambda}{t}((\frac{\mu +t}{2})^{n}-(\frac{\mu -t}{2})^{n}) $ and $B_{n}:=\frac{1}{t}((\frac{\mu +t}{2})^{n}-(\frac{\mu -t}{2})^{n})$.
		
		We only consider $A_n $ for $n\ge -1$ and $B_n $ for $n\ge 0$ in the following context. Note that we are Not concerned with the coefficients $A_n$($B_n$) when $t=0$ or the existence of $t$. We use it as a formal notation, what we actually use is  the explicit expression for each term just as given below. Then $A_mA_n+A_{m+n+1}=A_{m+1}B_{n+1}$( or $A_mB_n+B_{m+n+1}=B_{m+1}B_{n+1}$) without any assumption.
		
		For small $n$, we have\\
		$A_{0}=0$, $A_{1}=\lambda$,\\
		$A_{2}=\lambda \mu$, $A_{3}=\lambda (\mu^2-\lambda)$,\\
		$A_{4}=\lambda (\mu^3-2\mu\lambda)$. In particular, we assume that	$A_{-1}=-1$ even if $\lambda=0$.
	\end{rmk}

    	Let $\frakg$ be a vector space.  Let $$\frakm:=\mathrm{Hom}(\wedge\frakg,\frakg)\ \mathrm{and}\
    \fraka:=\Hom(\wedge\frakg,\frakg).$$
    For homogeneous elements	$f\in\Hom(\wedge^{n}\frakg,\frakg)$, $\sigma \in \rmS_n$ and $x_1,\dots,x_n\in \frakg$, we define $f\sigma^{-1}(x_1,\dots,x_n):=f(x_{\sigma(1)},\dots,x_{\sigma(n)})$.
    
	\begin{thm}\label{thm: Linfinity for extend Lie}
		Keeping the above notation and donote $t=\sqrt{\mu^2-4\lambda}$. There exists an $L_\infty[1]$-algebra structure on  $s \frakm\oplus\fraka$, where $l_i$ are given by 
		$$l_1(sf)=-A_n f$$
		$$
		l_2(sf,sg)      =    (-1)^{|f|} s[f,g]_{\NR}, $$
		and for $i\geq 2$,
		\begin{equation*}
			\begin{aligned}
				l_i(sf,\xi_1,\cdots,\xi_{i-1})  =&-\sum\limits_{\tau\in\Sh(m_1+1,\dots,m_{i-1}+1,n+2-i)}sgn(\tau)\\ &
				A_{n-i+1}\big(f(\xi_1\otimes\dots\otimes\xi_{i-1}\otimes\Id^{\otimes n+2-i})\big)\tau^{-1}\\
				&-\sum_{k=1}^{i-1}\sum\limits_{\tau\in\Sh(m_1+1,\dots,m_{i-1}+1,n+3-i,m_k)}sgn(\tau)(-1)^{(n+\sum\limits_{j=1 }^{k-1}m_j)m_k} \\ &
				B_{n-i+3} \xi_k  \big(f(\xi_1\otimes\dots\otimes\xi_{i-1}\otimes\Id^{\otimes n+3-i})\o \Id^{\otimes m_k}\big)\tau^{-1},	
			\end{aligned}
		\end{equation*}
		for homogeneous elements	$f\in\Hom(\wedge^{n+1}\frakg,\frakg)\subseteq \frakm$, $g\in\Hom(\wedge^{m+1}\frakg,\frakg)\subseteq \frakm$  and $\xi_j\in\Hom(\wedge^{m_j+1}\frakg,\frakg)\subseteq \fraka$, $1\leq j\leq i-1$, and all other components vanish.
	\end{thm}
	
	\begin{proof}
		First, we need to check that this is well-defined. This is given by the symmetry of $f,g$ and $\xi_j$, $1\leq j\leq i-1$.

		The rest of the proof will be given in Appendix $1$. Although somewhat involved, we can give a brief summary here:  the generalised Jacobi identity holds  by Lemma \ref{key lemma} and the fact that some terms  appear twice with opposite signs. 
		
	\end{proof}

	\begin{thm}\label{Thm: MC elements in ex Linifnity}
		Let $\frakg$ be a vector space.  Let $\pi\in\mathrm{Hom}(\wedge^2\frakg,\frakg)$, $T\in\mathrm{Hom}(\frakg,\mathfrak{g})$, then
		$(s \pi,T)\in \mathcal{MC}(s \frakm\oplus\fraka)$ if and only if $(\frakg,\pi,T)$ is  an extended Rota-Baxter Lie algebra  of weight  $(\mu,\lambda)$.
	\end{thm}

	\begin{proof}
		
		$(s\pi,T) \in\mathcal{MC}(s \frakm\oplus\fraka)$ if and only if
		$[\pi,\pi]_{\NR}=0$ (that is, $\pi$ is a Lie bracket) and
		\begin{equation*}
			\begin{aligned}
				0 &= 	l_1(s\pi)+\sum_{k=2}^\infty\frac{1}{(k-1)!}l_k(s\pi,\underbrace{T,\dots,T}_{(k-1)\ \mathrm{times}})  \\
				&=  	l_1(s\pi)+\mu  l_2(s\pi,T)+\frac{1}{2}l_3(s\pi,T,T)\\
				&=  -\lambda\pi-\mu T\bar{\circ} \pi+\frac{\lambda}{2}l_3(s\pi,T,T),
			\end{aligned}
		\end{equation*}
		For arbitrary  $(x,y)\in\wedge^2\frakg$, we have
		\begin{equation*}
			\begin{aligned}
				0&=	\big(-\lambda\pi- \mu T\circ \pi+\frac{1}{2}l_3(s\pi,T,T)\big)(x,y)\\
				&=-\lambda\pi(x,y)-\mu T\circ \pi(x,y)+\pi(Tx,Ty)-T\pi(x,Ty)-T\pi(Tx,y)\\
				&=-\lambda[x,y]-\mu T([x,y])-[T(x),T(y)]-T[T(x),y]- T[x,T(y)]\\
				&=[T(x),T(y)]-T[T(x),y]- T[x,T(y)]- \mu T([x,y])-\lambda[x,y],\\
			\end{aligned}
		\end{equation*}
		Hence $T$ is an extended Rota-Baxter operator of weight  $(\mu,\lambda)$.
	\end{proof}
	
	By Theorem \ref{thm: Linfinity for extend Lie}, we have
	
	\begin{cor}\label{curved dg Lie of operator}
		With the above notation. Let $(\frakg,\pi)$ be a Lie algebra, substitute $s\pi$ in place of $s f$ in the maps $l_i(sf,\xi_1,\cdots,\xi_{i-1})$ for $i\geq 1$.  Then $\fraka:=\Hom(\wedge\frakg,\frakg)$ equipped with $d_0(1)=l_1(s\pi), d_1(-)=l_2(s\pi,-),d_2(-,-)=l_3(s\pi,-,-)$ is a curved dg Lie algebra.  
	\end{cor}
    
    \medskip
    
	\subsection{Applications for special cases: Rota-Baxter Lie algebras with weight and modified Rota-Baxter Lie algebras}\ \label{subsec 5.2}
	We now discuss two significant algebraic structures: Rota-Baxter Lie algebras with weight and modified Rota-Baxter Lie algebras.
	
	\begin{cor}\label{wRBLie}
		Keeping the above notation and let $\lambda=0$, then there exists an $L_\infty[1]$-algebra structure on  $s \frakm\oplus\fraka$. It is just the  $L_\infty[1]$-algebra of  Rota-Baxter Lie algebras of weight $\mu$, where $l_i$ are given by

		$$
		l_2(sf,sg)      =    (-1)^{|f|} s[f,g]_{\NR}, $$
		and for $i\geq 2,n=i-2$,
		\begin{equation*}
			\begin{aligned}
				l_i(sf,\xi_1,\cdots,\xi_{i-1})=&\sum\limits_{\tau\in\Sh(m_1+1,\dots,m_{i-1}+1}\\ &
				sgn(\tau)\big(f(\xi_1\otimes\dots\otimes\xi_{i-1})\big)\tau^{-1}\\
				&-\sum_{k=1}^{i-1}\sum\limits_{\tau\in\Sh(m_1+1,\dots,m_{i-1}+1,m_k)}(-1)^{(n+\sum\limits_{j=1 }^{k-1}m_j)m_k} \\ 
				& sgn(\tau)\xi_k  \big(f(\xi_1\otimes\dots\otimes\xi_{i-1})\o \Id^{\otimes m_k}\big)\tau^{-1}
			\end{aligned}
		\end{equation*}
		for $i\geq 2,n>i-2$,
		\begin{equation*}
			\begin{aligned}
				l_i(sf,\xi_1,\cdots,\xi_{i-1})=
				&\sum_{k=1}^{i-1}\sum\limits_{\tau\in\Sh(m_1+1,\dots,m_{i-1}+1,n+3-i,m_k)}(-1)^{(n+\sum\limits_{j=1 }^{k-1}m_j)m_k} \\ 
				&-sgn(\tau)\mu^{n-i+2} \xi_k  \big(f(\xi_1\otimes\dots\otimes\xi_{i-1}\otimes\Id^{\otimes n+3-i})\o \Id^{\otimes m_k}\big)\tau^{-1},
			\end{aligned}
		\end{equation*}
		for homogeneous elements	$f\in\Hom(\wedge^{n+1}\frakg,\frakg)\subseteq \frakm$, $g\in\Hom(\wedge^{m+1}\frakg,\frakg)\subseteq \frakm$ and $\xi_j\in\Hom(\wedge^{m_j+1}\frakg,\frakg)\subseteq \fraka$, $1\leq j\leq i-1$, and all other components vanish.
	\end{cor}
	\begin{rmk}
		This controlling algebra  coincides with the $L_\infty[1]$-algebra of Rota-Baxter Lie algebras with weight by using the derived bracket technique in Appendix $2$.
	\end{rmk}
	
	\begin{cor}\label{mRNLie}
		Keeping the above notation and let $\mu=0$, then there exists an $L_\infty[1]$-algebra structure on  $s \frakm\oplus\fraka$. It is just the  $L_\infty[1]$-algebra of modified Rota-Baxter Lie algebras, where $l_i$ are given by :
		
		\begin{eqnarray}
			l_1(sf)=\frac{\sqrt{-\lambda}}{2}(\sqrt{-\lambda}^{n}-(-\sqrt{-\lambda})^{n}) f=
			\begin{cases}
			(-\lambda)^{\frac{n+1}{2}}	, &  \text{$n$ is odd }  \\
				0, & \text{$n$ is even},
			\end{cases}
		\end{eqnarray}
		$$
		l_2(sf,sg)      =    (-1)^{|f|} s[f,g]_{\NR}, $$
		and for $i\geq 2$,

		 when $n-i$ is even 
		\begin{equation*}
			\begin{aligned}
				l_i(sf,\xi_1,\cdots,\xi_{i-1})=&\sum\limits_{\tau\in\Sh(m_1+1,\dots,m_{i-1}+1,n+2-i)}\\ &
				sgn(\tau)(-\lambda)^{\frac{n-i+2}{2}}\big(f(\xi_1\otimes\dots\otimes\xi_{i-1}\otimes\Id^{\otimes n+2-i})\big)\tau^{-1}\\
				&+\sum_{k=1}^{i-1}\sum\limits_{\tau\in\Sh(m_1+1,\dots,m_{i-1}+1,n+3-i,m_k)}(-1)^{(n+\sum\limits_{j=1 }^{k-1}m_j)m_k} \\ 
				&sgn(\tau)(-\lambda)^{\frac{n-i+2}{2}} \xi_k  \big(f(\xi_1\otimes\dots\otimes\xi_{i-1}\otimes\Id^{\otimes n+3-i})\o \Id^{\otimes m_k}\big)\tau^{-1},    \\ 
			\end{aligned}
		\end{equation*}\\
		when $n-i$ is odd, $l_i(sf,\xi_1,\cdots,\xi_{i-1})=0$,
		for homogeneous elements	$f\in\Hom(\wedge^{n+1}\frakg,\frakg)\subseteq \frakm$, $g\in\Hom(\wedge^{m+1}\frakg,\frakg)\subseteq \frakm$	and $\xi_j\in\Hom(\wedge^{m_j+1}\frakg,\frakg)\subseteq \fraka$, $1\leq j\leq i-1$, and all other components vanish.
	\end{cor}

    \bigskip
    
	\section{The cohomology of extended Rota-Baxter Lie algebras}\label{sec 6}

	Let $\frakg$ be a vector space.  Let $$\frakm:=\mathrm{Hom}(\wedge\frakg,\frakg)\ \mathrm{and}\
	\fraka:=\Hom(\wedge\frakg,\frakg).$$
	Recall that we have constructed an $L_\infty[1]$-structure on $s\frakm\oplus\fraka$ in Theorem~\ref{thm: Linfinity for extend Lie}.

	Let $(\frakg, \pi, T)$ be an extended Rota-Baxter Lie algebra.  By Theorem~\ref{Thm: MC elements in ex Linifnity}, $(s\pi, T)$ is a Maurer-Cartan element in the $L_\infty[1]$-algebra $s\frakm\oplus\fraka$. By  Proposition~\ref{Prop: twist-L-infty[1]}, twisting $s\frakm\oplus\fraka$ by  $(s\pi, T)$ gives a new $L_\infty[1]$-algebra, whose new differential is denoted by $l_{1}^{(s\pi,T)}$.

	\begin{prop}\label{cochaincomplexad}
		The  complex   $(s\frakm\oplus\fraka,l_{1}^{(s\pi,T)})$ is given by a mapping cone.
	\end{prop}
	\begin{proof}

		It suffices to make explicit the differential $l_1^{(s\pi, T)}$.
		
		For  $n\geq 1$, $f\in \mathrm{Hom}( \wedge^n\frakg,\frakg), g\in  \mathrm{Hom}( \wedge^{n-1}\frakg,\frakg)$,
		\begin{eqnarray*}
			l_1^{(s\pi, T)}(sf, g)
			&=& \sum_{k=0}^{\infty}\frac{1}{k!}l_{k+1}(\underbrace{(s\pi,T),\cdots,(s\pi,T)}_{k\ \mathrm{times}}, (sf, g))\\
			&=&l_1(sf,g)+ l_2((s\pi, T),(sf,g))
			+ \sum_{k=2}^{\infty}\frac{1}{k!}l_{k+1}(\underbrace{(s\pi,T),\cdots,(s\pi,T)}_{k\ \mathrm{times}}, (sf, g))\\
			&=& \big(l_2(s\pi, sf), l_1(sf)+l_2(s\pi, g)+l_3(s\pi, T, g)+l_2(sf, T)
			+ \sum_{k=2}^{n}\frac{1}{k!}l_{k+1}(sf, \underbrace{T,\cdots,T}_{k\ \mathrm{times}}) \big).
		\end{eqnarray*}

		It is easy to see that
		$l_2(s\pi, sf)=- s[\pi, f]_{\NR}$ induces a  differential on $\frakm$. Denote  $\partial_{\Lie}:=[\pi,-]_{\NR}$. Then $(\frakm,\partial_\Lie)$ is just the cochain complex of the Lie algebra $\frakg$ in Definition \ref{cohomology of algebra}.

		Let us compute $l_1(sf)+l_2(s\pi, g)+l_3(s\pi, T, g)
		+l_2(sf, T)+ \sum\limits_{k=2}^{n}\frac{1}{k!}l_{k+1}(sf, \underbrace{T,\cdots,T}_{k\ \mathrm{times}})$.
		
		For   $x_1,\dots,x_{n}\in \frakg$, note that $(l_2(s\pi, g)+l_3(s\pi, T, g))(x_1, \dots, x_{n})$
		induces a  differential on $\fraka$ just as in Proposition-Definition \ref{cohomology of operator}. Denote this differential by $\partial_\RB$. So $(\fraka,\partial_\RB)$ is just the cochain complex of $T$.

		On the other hand,
		\begin{eqnarray*}
			&&l_1(sf)+l_2(sf, T)+\sum\limits_{k=2}^{n}\frac{1}{k!}l_{k+1}(sf, \underbrace{T,\cdots,T}_{k\ \mathrm{times}})(x_1, \dots, x_{n})\\
			&=&-A_{n-1}f-A_{n-2}f\bar{\circ} T-B_{n}T\bar{\circ} f\\
			&&-\sum_{k=2}^{n}\frac{1}{k!} \sum\limits_{\tau\in\Sh(1,\dots,1,n-k)}sgn(\tau)A_{n-k-1}\big(f(T\otimes\dots\otimes T\otimes\Id^{\otimes n-k})\big)\tau^{-1}(x_1, \dots, x_{n})\\
			&&-\sum_{k=2}^{n}\frac{1}{k!} sgn(\tau)B_{n-k+1} \sum_{i=1}^{k}\sum\limits_{\tau\in\Sh(1,\dots,1,n+1-k)}
			T  \big(f(T\otimes\dots\otimes T\otimes\Id^{\otimes n+1-k})\big)\tau^{-1} (x_1, \dots, x_{n})\\
			&=& -\sum_{\tau\in \Sh(k,n-k),0\le k\le n}sgn(\tau)A_{n-k-1}f(T(x_{i_1}),\cdots,T(x_{i_k}),x_1, \cdots, \hat{x_{i_1}},\cdots, \hat{x_{i_k}},\cdots,x_n ))\\
			&&-\sum_{\tau\in \Sh(k-1,n-k+1),1\le k\le n}sgn(\tau)B_{n-k+1}T\circ f(T(x_{i_1}),\cdots,T(x_{i_{k-1}}),x_1, \cdots, \hat{x_{i_1}},\cdots, \hat{x_{i_{k-1}}},\cdots,x_n ),\\
		\end{eqnarray*}
	  where $i_j,j\ge 1$ are decided by $\tau$.
	  
		It induces a linear map $ \frakm\to \fraka$, denote by $\delta$.
		That is, we have
		$$ l_1^{(s\pi, T)}(sf, g) =\big(-s\partial_\Lie^{n}(f), \partial_\RB^{n-1}(g)+\delta^n(f)\big)$$ and $\delta$ is a cochain map.
	\end{proof}
	
	From the above proof, we could see that:

	\begin{defn}\label{RBLie homology}
		With the above notation, \textbf{the cochain complex of the extended Rota-Baxter Lie algebra} $(\frakg,\pi,T)$ is defined to be $( \rmC^*_{\RBA}(\frakg):=s\frakm\oplus \fraka,\partial_\RBA)$ where $\partial_\RBA^{n}(sf,g)=\big(s\partial_\Lie^{n}(f), -\partial_\RB^{n-1}(g)-\delta^n(f)\big)$. It's induced by the mapping cone of the cochain map $\delta$ from the cochain complex of Lie algebra $\frakg$ to the cochain complex of operator $T$.  The corresponding cohomology group is called \textbf{the cohomology of the extended Rota-Baxter Lie algebra $\frakg$}.
	\end{defn}
	
	We generalize the above definition to the case of coefficients in an  arbitrary representation.
	
	Let $(V,T_V)$ be a representation of the extended Rota-Baxter Lie algebra $(\frakg,\pi,T)$ of weight $(\mu,\lambda)$.

	Consider the trivial extension  $ \frakg \ltimes V$ of the extended Rota-Baxter Lie algebra $(\frakg,T_{\frakg})$ by the representation $(V,T_V)$ in Proposition~\ref{Prop: trivial extension of Rota-Baxter rep}, then we get the complex $( \rmC^*_{\RBA}(\frakg \ltimes V),\partial_{\RBA}^*)$ by  Definition \ref{RBLie homology}.

	Given two vector spaces $V$ and $W$, by the isomorphism
	$$\begin{array}{rcl}
		\varphi:\wedge^n(V\oplus W)	& \longrightarrow     &\bigoplus_{k+l=n, k, l\geq 0}\wedge^kV\otimes \wedge^lW \\
		(v_1+w_1)\wedge\cdots\wedge(v_n+w_n)        &\mapsto      &\sum\limits_{\sigma\in \Sh(k,n-k)}sgn(\sigma)v_{\sigma(1)}\wedge\cdots\wedge v_{\sigma(k)}\otimes w_{\sigma(k+1)}\wedge\cdots\wedge w_{\sigma(n)},
	\end{array}$$
	we have an isomorphism\label{subspace}
	\begin{align}
		\mathrm{Hom}(\wedge^n(V\oplus W),V\oplus W)\cong\big(\oplus_{k+l=n}\mathrm{Hom}(\wedge^kV\otimes \wedge^lW,V)\oplus(\oplus_{k+l=n}\mathrm{Hom}(\wedge^kV\otimes \wedge^lW,  W))\big).\end{align}
	With this isomorphism in mind, we   have  the following result. 
	
	\begin{prop-def} \label{RBLie homology on V}
		\textbf{The cochain complex of the extended Rota-Baxter Lie algebra $(\frakg,\pi,T)$ with coefficients in the representation $(V,T_V)$}, denoted by $(\rmC_{\RBA}^*(\frakg,V),\partial_{\RBA}^*)$, is the subcomplex of
		$( \rmC^*_{\RBA}(\frakg \ltimes V, (\frakg \ltimes V)_\ad),\partial_{\RBA}^*)$ given by  the natural injection via the isomorphism \ref{subspace}
		from $$\rmC_{\RBA}^n(\frakg, V)=\mathrm{Hom}( \wedge^n\frakg,V)\oplus \mathrm{Hom}( \wedge^{n-1}\frakg,V)$$ to $$\rmC^n_{\RBA}(\frakg \ltimes V, (\frakg \ltimes V)_\ad)=\mathrm{Hom}( \wedge^n(\frakg \oplus  V),\frakg \oplus V)\oplus \mathrm{Hom}( \wedge^{n-1}(\frakg \oplus V),\frakg \oplus V).$$ The corresponding cohomology group is called \textbf{the cohomology of the extended Rota-Baxter Lie algebra $(\frakg,T)$ with coefficients in the representation $(V,T_V)$}. 
	\end{prop-def}
	
 \begin{rmk}\label{cocycle}
 	Now we provide a more detailed description of the cochain complex of  extended Rota-Baxter Lie algebras. Likewise, we can generalize the cohomology of Lie algebras and extended Rota-Baxter operators to the case of coefficients in an arbitrary representation.
 	
 	Let $(\frakg,\pi)$ be a Lie algebra and  $V$ be a  representation over $\frakg$. \textbf{The cochain complex of $\frakg$ with coefficients in $V$} is defined to be $(\rmC^*_{\Lie}(\frakg,V):=\mathrm{Hom}(\bigwedge (\frakg),V),\partial_{\Lie})$ where the coboundary operator $\partial_{\Lie}$ is given as follows. The corresponding cohomology is called \textbf{the cohomology of $\frakg$ with coefficients in $V$}. We could see that this is just the Chevalley-Eilenberg cochain complex and Chevalley-Eilenberg cohomology of $\frakg$ with coefficients in $V$. \\
 	
 	The coboundary operator $$\partial_{\Lie}^n: \mathrm{Hom}(\wedge^n \frakg,V)\longrightarrow  \mathrm{Hom}(\wedge^{n+1} \frakg,V), n\geq 0$$ is given by
 	\[\begin{split}
 		\partial_{\Lie}^n (f)(x_1,\dots,x_{n+1})&=\sum_{i=1}^{n+1}(-1)^{i+n}\rho(x_i) f(x_1,\dots,\hat{x}_i, \dots, x_{n+1})\\
 		&+\sum_{1\leq i<j\leq n+1} (-1)^{i+j+n+1}f([x_i, x_j], x_1,\dots,\hat{x}_i, \dots,\hat{x}_j, \dots,x_{n+1}),
 	\end{split}\]
 	for all $f\in \mathrm{Hom}(\wedge^n \frakg,V),~x_1,\dots, x_{n+1}\in \frakg$, where $\hat{x_i}$ means deleting this element.

 	Consider an extended Rota-Baxter Lie algebra $(\frakg,\pi,T)$ and $(V, T_V)$ is an extended Rota-Baxter representation over $(\frakg,\pi,T)$. \textbf{The cochain complex of the extended Rota-Baxter operator  $T$ with coefficients in $V$} denoted by $(\rmC^*_{\RB}(\frakg,V):=\mathrm{Hom}(\bigwedge (\frakg),V),\partial_{\RB})$ is defined to be the Chevalley-Eilenberg cochain complex   of $\frakg_\star$ with coefficient in $V_t$(See Proposition  \ref{Prop: new RB algebra} and \ref{descending property}). The corresponding  cohomology group is called \textbf{the cohomology of the extended Rota-Baxter operator  $T$ with coefficients in $V$}.
 	
 	There exists a cochain map $\delta: \rmC^*_{\Lie}(\frakg,V)\to \rmC^*_{\RB}(\frakg,V)$ defined by
 	\begin{align*}
 		&\delta(f)(x_1,\cdots,x_n)\\
 		=&-\sum_{\tau\in \Sh(k,n-k),k\ge 0}sgn(\tau)A_{n-k-1}f(T(x_{i_1}),\cdots,T(x_{i_k}),x_1, \cdots, \hat{x_{i_1}},\cdots, \hat{x_{i_k}},\cdots,x_n )\\
 		&-\sum_{\tau\in \Sh(k-1,n-k+1),k\ge 1}sgn(\tau)B_{n-k+1}T_V\circ f(T(x_{i_1}),\cdots,T(x_{i_{k-1}}),x_1, \cdots, \hat{x_{i_1}},\cdots, \hat{x_{i_{k-1}}},\cdots,x_n ),
 	\end{align*}
 	for $f\in \mathrm{Hom}(\wedge^n (\frakg),V)$ and $x_1,\dots,x_n \in \frakg$.
 	
 	Then the  coboundary operator of  cochain complex of the extended Rota-Baxter Lie algebra $(\frakg,\pi,T)$ with coefficients in the  $(V,T_V)$ is just $\partial_{\RBA}(sf,g)=\big(s\partial_\Lie^{n}(f), -\partial_\RB^{n-1}(g)-\delta^n(f)\big)$ for $f\in \mathrm{Hom}( \wedge^n\frakg,V), g\in  \mathrm{Hom}( \wedge^{n-1}\frakg,V)$.

 	 We compute $n$-cocycles of $\rmC_{\RBA}^*(\frakg,V)$ for small $n$.

 	 For all $(f,v)\in\Hom (\frakg,V)\oplus V$, $\partial^1_{\RBA} (f,v)=0$ if and only if $\partial^1_\Lie  f=0$ and
 	 \begin{align}\label{1-cocycle}
 	 [T(x), v]-T_V[x,v]=f(T(x))-T_V(f(x)),\quad \forall x\in \frakg.
 	\end{align}
 	 
 	 For all $(f,g)\in\Hom (\wedge^2\frakg,V)\oplus\Hom (\frakg,V) $,  $\partial^2_{\RBA}(f,g)=0$ if and only if $\partial^2_\Lie f=0,$ and
 	 \begin{align}\label{2-cocycle}
 	 	&-[T(x), g(y)]+T_V[x, g(y)]+ [T(y), g(x)]-T_V[y, g(x)]+g([x, y]_\star)\\
 	 	= & -\lambda f(x,y)+f(T(x),T(y))-\mu T_Vf(x,y)-T_V(f(T(x),y))-T_V(f(x,T(y))),\notag
 	 \end{align}
 	 for all $x,y\in \frakg.$

 \end{rmk}
	
	\bigskip
	
	\section{Formal deformations of extended Rota-Baxter  Lie algebras}\label{sec 7}
	
	In this section, we will study formal deformations of extended Rota-Baxter Lie algebras and interpret  them  via  the  low-degree   cohomology groups  of extended Rota-Baxter Lie algebras defined in Proposition-Definition \ref{RBLie homology on V}.

	Let $(\frakg, T)$ be an extended Rota-Baxter Lie algebra of weight $(\mu,\lambda)$.   Consider a 1-parameter family:
	\[\mu_t=\sum_{i=0}^\infty \mu_it^i, \ \mu_i\in \mathrm{Hom}( \wedge^2\frakg,\frakg),\quad  T_t=\sum_{i=0}^\infty T_it^i,  \ T_i\in \mathrm{Hom}(\frakg,\frakg).\]
	
	\begin{defn}
		A  1-parameter formal deformation of an extended Rota-Baxter Lie algebra $(\frakg, \pi,T)$ is a pair $(\mu_t,T_t)$ which endows the free $\bfk[[t]]$-module $\frakg[[t]]$ with an extended Rota-Baxter Lie algebra structure over $\bfk[[t]]$ such that $T_0=T$ and $\mu_0=\pi$.
	\end{defn}
	
	The power series $\mu_t$ and $ T_t$ determine a  1-parameter formal deformation of the extended Rota-Baxter Lie algebra $(\frakg,\pi,T)$ if and only if for any $x,y,z\in \frakg$, the following equations hold :
	\begin{eqnarray*}
			0 &=&  \mu_t(x, \mu_t(y, z))+\mu_t(y, \mu_t(z, x))+\mu_t(z, \mu_t(x, y)),\\\\
		\mu_t(T_t(x), T_t(y))&=& T_t\Big(\mu_t(x, T_t(y))+\mu_t(T_t(x), y)+\mu \mu_t(x,y)\Big)+\lambda \mu_t(x,y).
	\end{eqnarray*}
	By expanding these equations and comparing the coefficient of $t^n$, we obtain  that $\{\mu_i\}_{i\geqslant0}$ and $\{T_i\}_{i\geqslant0}$ have to  satisfy: for any $n\geqslant 0$,
	\begin{equation}\label{Eq: deform eq for  products in RBA}
		0=\sum_{i=0}^n \mu_i(x, \mu_{n-i}(y, z))+\mu_i(y, \mu_{n-i}(z, x))+\mu_i(z, \mu_{n-i}(x, y)),\end{equation}
	\begin{equation}\label{Eq: Deform RB operator in RBA} \begin{array}{rcl}
			\sum_{i+j+k=n\atop i, j, k\geqslant 0}	\mu_{i}\circ(T_j\ot T_{k})&=&\sum_{i+j+k=n\atop i, j, k\geqslant 0} T_{i}\circ \mu_j\circ (\id\ot T_{k})\\
			&  &+\sum_{i+j+k=n\atop i, j, k\geqslant 0} T_{i}\circ\mu_j\circ (T_{k}\ot \id)+\mu\sum_{i+j=n\atop i, j \geqslant 0}T_i\circ\mu_{j}+\lambda \mu_n.
	\end{array}\end{equation}
	Obviously, when $n=0$, it is just the extended Rota-Baxter Lie algebra $(\frakg,\pi,T)$.
	
	\smallskip
	
	\begin{prop}\label{Prop: Infinitesimal is 2-cocycle}
		Let $(\frakg[[t]],\mu_t,T_t)$ be a  1-parameter formal deformation of the
		extended Rota-Baxter Lie algebra $(\frakg,\pi,T)$ of weight $(\mu,\lambda)$. Then
		$(\mu_1,T_1)$ is a 2-cocycle in the cochain complex
		$\rmC_{\RBA}^\bullet(\frakg)$.
	\end{prop}
	\begin{proof} When $n=1$,   Equations~(\ref{Eq: deform eq for  products in RBA}) and (\ref{Eq: Deform RB operator in RBA})  become
		$$0=\pi(x,\mu_1(y,z))+\pi(y,\mu_1(z,x))+\pi(z,\mu_1(x,y))+\mu_1(x,\pi(y,z))+\mu_1(y,\pi(z,x))+\mu_1(z,\pi(x,y))$$
		and
		$$\begin{array}{cl}
			&\mu_1 (T\ot T)-\{T\circ\mu_1\circ(\id\ot T)+T\circ\mu_1\circ(T\ot \id)+\mu T\circ \mu_1+\lambda \mu_1\}\\
			=&T\circ \pi\circ(\id\ot T_1)+T\circ \pi\circ(T_1\ot\id)
			+T_1\circ \pi\circ(\id\ot T)+T_1\circ \pi\circ(T\ot \id)+\mu T_1\circ \pi\\
			&-\pi\circ(T_1\ot T)-\pi\circ(T\ot T_1),
		\end{array}$$
		Note that  the first equation is exactly $\partial_{\Lie}^2(\mu_1)=0\in \C^\bullet_{\Lie}(\frakg)$ and that  second equation is exactly  \[\delta^2(\mu_1)=-\partial_{\RB}^1(T_1) \in \C^\bullet_{\RB}(\frakg).\]
		So $(\mu_1,T_1)$ is a 2-cocycle in $\C^\bullet_{\RBA}(\frakg)$, see Eq. (\ref{2-cocycle}).
	\end{proof}
	
	\smallskip
	
	\begin{defn} The 2-cocycle $(\mu_1,T_1)$ is called the infinitesimal of the 1-parameter formal deformation $(\frakg[[t]],\mu_t,T_t)$ of the extended Rota-Baxter Lie algebra $(\frakg,\pi,T)$.
	\end{defn}

	In general, we can rewrite Equations~(\ref{Eq: deform eq for  products in RBA}) and (\ref{Eq: Deform RB operator in RBA}) as
	\begin{eqnarray} \label{Eq: general formal of deform product in RBA} \partial_{\Lie}^2(\mu_n) = \frac{1}{2}\sum_{i=1}^{n-1} [\mu_i, \mu_{n-i}]_{\NR}\end{eqnarray}
	\begin{equation} \label{Eq: general formal of deform RBO in RBA}
		\begin{array}{rcl}\partial_{\RB}^{1}(T_n)   +\delta^2(\mu_n)&=& \sum_{i+j+k=n\atop 0 \leqslant i, j, k\leqslant n-1}	\mu_{i}\circ(T_j\ot T_{k})-\sum_{i+j+k=n\atop 0 \leqslant i, j, k\leqslant n-1} T_{i}\circ \mu_j\bar{\circ} (\id\ot T_{k})\\ &&-\sum_{i+j+k=n\atop 0 \leqslant i, j, k\leqslant n-1} T_{i}\circ\mu_j\bar{\circ} (T_{k}\ot \id)-\sum_{i+j=n\atop 0 \leqslant i, j\leqslant n-1}T_i\circ\mu_{j}.
	\end{array}\end{equation}
	
	\smallskip
	\begin{defn}
		Let $(\frakg[[t]],\mu_t,T_t)$ and $(\frakg[[t]],\mu_t',T_t')$ be two 1-parameter formal deformations of the extended Rota-Baxter Lie algebra $(\frakg,\pi,T)$. A formal isomorphism from $(\frakg[[t]],\mu_t',T_t')$ to $(\frakg[[t]], \mu_t, T_t)$ is a power series $\psi_t=\sum_{i=0}\psi_it^i: \frakg[[t]]\rightarrow \frakg[[t]]$, where $\psi_i: \frakg\rightarrow \frakg$ are linear maps with $\psi_0=\id_\frakg$, such that:
		\begin{eqnarray}\label{Eq: equivalent deformations}\psi_t\circ \mu_t' &=& \mu_t\circ (\psi_t\ot \psi_t),\\
			\psi_t\circ T_t'&=&T_t\circ\psi_t. \label{Eq: equivalent deformations2}
		\end{eqnarray}
		In this case, we say that 1-parameter formal deformations $(\frakg[[t]], \mu_t,T_t)$ and
		$(\frakg[[t]],\mu_t',T_t')$ are  equivalent.
	\end{defn}

	\smallskip

	Given an extended Rota-Baxter Lie algebra $(\frakg,\pi,T)$, the power series $\mu_t,T_t$
	with $\mu_i=0, T_i=0$ for $i\ge 1$ make
	$(\frakg[[t]],\mu_t,T_t)$ into a $1$-parameter formal deformation of
	$(\frakg,\pi,T)$. The formal deformations equivalent to this one are called \textbf{trivial}.
	\smallskip
	
	\begin{thm}
		The infinitesimals of two equivalent 1-parameter formal deformations of $(\frakg,\pi,T)$ are in the same cohomology class in $\rmH^2_{\RBA}(\frakg)$.
	\end{thm}
	
	\begin{proof} Let $\psi_t:(\frakg[[t]],\mu_t',T_t')\rightarrow (\frakg[[t]],\mu_t,T_t)$ be a formal isomorphism.
		Expanding the identities and collecting coefficients of $t$, we get from Equations~(\ref{Eq: equivalent deformations}) and (\ref{Eq: equivalent deformations2}):
		\begin{eqnarray*}
			\mu_1'&=&\mu_1+\pi\circ(\id\ot \psi_1)-\psi_1\circ\pi+\pi\circ(\psi_1\ot \id),\\
			T_1'&=&T_1+T\circ\psi_1-\psi_1\circ T,
		\end{eqnarray*}
		that is, we have\[(\mu_1',T_1')-(\mu_1,T_1)=(\partial_{\Lie}^1(\psi_1), -\delta^1(\psi_1))=\partial_{\RBA}^1(\psi_1,0)\in  \C^2_{\RBA}(\frakg),\] see Eq. (\ref{1-cocycle}).
	\end{proof}
	
	\smallskip
	
	\begin{defn}
		An extended Rota-Baxter Lie algebra $(\frakg,\pi,T)$ is said to be rigid if every 1-parameter formal deformation is trivial.
	\end{defn}
	
	\begin{thm}
		Let $(\frakg,\pi,T)$ be an extended Rota-Baxter Lie algebra of weight $(\mu,\lambda)$. If $\rmH^2_{\RBA}(\frakg)=0$, then $(\frakg,\pi,T)$ is rigid.
	\end{thm}
	
	\begin{proof}Let $(\frakg[[t]], \mu_t, T_t)$ be a $1$-parameter formal deformation of $(\frakg,\pi,T)$. By Proposition~\ref{Prop: Infinitesimal is 2-cocycle},
		$(\mu_1, T_1)$ is a $2$-cocycle. By $\rmH^2_{\RBA}(\frakg)=0$, there exists a $1$-cochain $$(\psi_1', x) \in \C^1_\RBA(\frakg)= C^1_{\Lie}(\frakg)\oplus \Hom(\bfk, \frakg)$$ such that
		$(\mu_1, T_1) =  \partial_{\RBA}(\psi_1', x), $
		that is, $\mu_1=\partial_{\Lie}^1(\psi_1')$ and $T_1=-\partial_{\RB}^0(x)-\delta^1(\psi_1')$. Let $\psi_1=\psi_1'+\partial_{\Lie}^0(x)$. Then
		$\mu_1= \partial_{\Lie}^1(\psi_1)$ and $T_1=-\delta^1(\psi_1)$, as it can be readily seen that $\delta^1(\partial_{\Lie}^0(x))=\partial_{\RB}^0(x)$.
		
		Setting $\psi_t = \Id_\frakg -\psi_1t$, we have a deformation $(\frakg[[t]], \overline{\mu}_t, \overline{T}_t)$, where
		$$\overline{\mu}_t=\psi_t^{-1}\circ \mu_t\circ (\psi_t\times \psi_t)$$
		and $$\overline{T}_t=\psi_t^{-1}\circ T_t\circ \psi_t.$$
		It can be easily verified  that $\overline{\mu}_1=0, \overline{T}_1=0$. Then
		$$\begin{array}{rcl} \overline{\mu}_t&=& \pi+\overline{\mu}_2t^2+\cdots,\\
			\overline{T}_t&=& T+\overline{T}_2t^2+\cdots.\end{array}$$
		By Equations~(\ref{Eq: general formal of deform product in RBA}) and (\ref{Eq: general formal of deform RBO in RBA}), we see that $(\overline{\mu}_2,  \overline{T}_2)$ is still a $2$-cocycle, so by induction, we can show that
		$ (\frakg[[t]], \mu_t , T_t) $ is equivalent to the trivial  deformation $(\frakg[[t]], \pi, T).$
		Thus, $(\frakg,\pi,T)$ is rigid.
		
	\end{proof}
	
	\bigskip

	\section{Abelian extensions of extended Rota-Baxter Lie algebras}\label{sec 8}
	
	In this section, we study abelian extensions of extended Rota-Baxter Lie algebras and show that they are classified by the second cohomology group.

	Notice that a vector space $V$ together with a linear transformation $T_V:V\to V$ is naturally an extended Rota-Baxter Lie algebra where the bracket on $V$ is defined to be $[u,v]=0$ for all $u,v\in V.$
	
	\begin{defn}
		An   abelian extension  of extended Rota-Baxter Lie algebras is a short exact sequence of  morphisms of extended Rota-Baxter Lie algebras
		\begin{eqnarray}\label{Eq: abelian extension} 0\to (V, T_V)\stackrel{i}{\to} (\hat{\frakg}, \hat{T})\stackrel{p}{\to} (\frakg, T)\to 0,
		\end{eqnarray}
		that is, there exists a commutative diagram:
		\[\begin{CD}
			0@>>> {V} @>i >> \hat{\frakg} @>p >> \frakg @>>>0\\
			@. @V {T_V} VV @V {\hat{T}} VV @V T VV @.\\
			0@>>> {V} @>i >> \hat{\frakg} @>p >> \frakg @>>>0,
		\end{CD}\]
		where the extended Rota-Baxter Lie algebra $(V, T_V)$	satisfies  $[u,v]=0$ for all $u,v\in V.$
		
		We will call $(\hat{\frakg},\hat{T})$ an abelian extension of $(\frakg,\pi,T)$ by $(V,T_V)$.
	\end{defn}

	\begin{defn}
		Let $(\hat{\frakg}_1,\hat{T}_1)$ and $(\hat{\frakg}_2,\hat{T}_2)$ be two abelian extensions of $(\frakg,\pi,T)$ by $(V,T_V)$. They are said to be  isomorphic  if there exists an isomorphism of extended Rota-Baxter Lie algebras $\zeta:(\hat{\frakg}_1,\hat{T}_1)\rar (\hat{\frakg}_2,\hat{T}_2)$ such that the following commutative diagram holds:
		\begin{eqnarray}\label{Eq: isom of abelian extension}\begin{CD}
				0@>>> {(V,T_V)} @>i >> (\hat{\frakg}_1,{\hat{T}_1}) @>p >> (\frakg,\pi,T) @>>>0\\
				@. @| @V \zeta VV @| @.\\
				0@>>> {(V,T_V)} @>i >> (\hat{\frakg}_2,{\hat{T}_2}) @>p >> (\frakg,\pi,T) @>>>0.
		\end{CD}\end{eqnarray}
	\end{defn}
	
	A   section of an abelian extension $(\hat{\frakg},{\hat{T}})$ of $(\frakg,\pi,T)$ by $(V,T_V)$ is a linear map $s:\frakg\rar \hat{\frakg}$ such that $p\circ s=\Id_\frakg$. We identify $V$ with $i(V).$
	
	\bigskip
	
	Let    $(\hat{\frakg},\hat{T})$ be  an abelian extension of $(\frakg,\pi,T)$ by $(V,T_V)$ having the form \eqref{Eq: abelian extension}. Choose a section $s:\frakg\rar \hat{\frakg}$. We   define
	$$
	[a,m]:=[s(a),m]_{\hat{\frakg}}, \quad \forall a\in \frakg, m\in V.
	$$
	When there is no ambiguity, we will omit the subscript of the bracket.
	\begin{prop}\label{Prop: new RB representations from abelian extensions}
		With the above notation, $(V, T_V)$ is an extended Rota-Baxter  representation over $(\frakg,\pi,T)$.
	\end{prop}
	\begin{proof}
		For any $a,b\in \frakg,\,m\in V$, since $s([a,b])-[s(a),s(b)]\in V$ implies $[s([a,b]),m]=[[s(a),s(b)],m]$, we have
		\[[ [a,b] , m]=[s([a,b]),m]=[[s(a),s(b)],m]\xlongequal{\mbox{Jacobi identity}}[a,[b,m]]-[b,[a,m]].\]
		Hence,  this gives a representation structure.

		Moreover, ${\hat{T}}(s(a))-s(T(a))\in V$ means that  $[{\hat{T}}(s(a)),m]=[s(T(a)),m]$. Thus we have
		\begin{align*}
			[T(a),T_V(m)]&=[s(T(a)),T_V(m)]\\
			&=[\hat{T}(s(a)),T_V(m)]\\
			&=\hat{T}([\hat{T}(s(a)),m]+[s(a),T_V(m)]+\mu [s(a),m])+\lambda [s(a),m] \\
			&=T_V([T(a),m]+[a,T_V(m)]+\mu[ a,m])+\lambda [a,m]
		\end{align*}

		Hence, $(V, T_V)$ is an extended  Rota-Baxter  representation over $(\frakg,\pi,T)$.
	\end{proof}
	
	We  further  define linear maps $\psi:\frakg\ot \frakg\rar V$ and $\chi:\frakg\rar V$ respectively by
	\begin{align*}
		\psi(a,b)&=[s(a),s(b)]-s([a,b]),\quad\forall a,b\in \frakg,\\
		\chi(a)&={\hat{T}}(s(a))-s(T(a)),\quad\forall a\in \frakg.
	\end{align*}

	\begin{prop}\label{psi and chi}
		The pair
		$(\psi,\chi)$ is a 2-cocycle  of the  extended Rota-Baxter Lie algebra $(\frakg,\pi,T)$ with  coefficients  in the extended Rota-Baxter representation $(V,T_V)$ introduced in Proposition~\ref{Prop: new RB representations from abelian extensions}.
	\end{prop}

   	\begin{proof}
   	
   	To see $(\psi,\chi)$ is a  2-cocycle, it suffices to check the identities in Eq. (\ref{2-cocycle}).
   	
   	For the first identity, 
   	\begin{align*}
   		&-[a,\psi(b\ot c)]+[b,\psi(a\ot c)]-[c,\psi(a\ot b)]+\psi([a,b]\ot c)-\psi([a,c]\ot b)+\psi([b,c]\ot a)\\
   		=&-[a,([s(b),s(c)]-s[(b,c)])]+[b,[s(a),s(c)]-s([a,c])]-[c,[s(a),s(b)]-s([a,b])]\\
   		&+[s([a,b]),s(c)]-s([[a,b],c])-[s([a,c]),s(b)]-s([[a,c],b])+[s([b,c]),s(a)]-s([[b,c],a])\\
   		=&0
   	\end{align*}
   	Thus $\partial^2_\Lie (\psi)=0$.
   	
   	For the second identity, 
   	\begin{align*}
   		&[T(a),\chi(b)]+[\chi(a),T(b)]+\psi(T(a)\ot T(b))\\
   		=&[s(T(a)),({\hat{T}}(s(b))-s(T(b)))]+[({\hat{T}}(s(a))-s(T(a))),s(T(b))]+[s(T(a)),s(T(b))]-s([T(a),T(b)])\\
   		=&[\hat{T}(s(a)),({\hat{T}}(s(b))-s(T(b)))]+[({\hat{T}}(s(a))-s(T(a))),\hat{T}(s(b))]+[s(T(a)),s(T(b))]-s([T(a),T(b)])\\
   		=&2[\hat{T}(s(a)),\hat{T}(s(b))]-[\hat{T}(s(a)),s(T(b))]-[s(T(a)),\hat{T}(s(b))]+[s(T(a)),s(T(b))]-s([T(a),T(b)])\\
   		=&2[\hat{T}(s(a)),\hat{T}(s(b))]-[\hat{T}(s(a)),s(T(b))]-[\hat{T}(s(a)),(\hat{T}(s(b))-s(T(b)))]-s([T(a),T(b)])\\
   		=&[\hat{T}(s(a)),\hat{T}(s(b))]-s([T(a),T(b)])
   	\end{align*}
   	
   	\begin{align*}
   		&T_M([\chi(a),b])+T_M(\psi(T(a)\ot b))+\chi([T(a),b])+T_M([a,\chi(b)])+T_M(\psi(a\ot T(b)))+\chi([a,T(b)])\\
   		&+\mu T_M(\psi(a\ot b))+\lambda \psi(a\ot b)+\mu \chi([a,b])\\
   		=&\hat{T}([(\hat{T}(s(a))-s(T(a))),s(b)])+\hat{T}([s(T(a)),s(b)]-s([T(a),b]))+{\hat{T}}(s([T(a),b]))-s(T([T(a),b]))\\
   		&+\hat{T}([s(a),({\hat{T}}(s(b))-s(T(b)))])+\hat{T}([s(a),s(T(b))]-s([a,T(b)]))+{\hat{T}}(s([a,T(b)]))-s(T([a,T(b)]))\\
   		&+\mu \hat{T}([s(a),s(b)]-s([a,b]))+\lambda([s(a),s(b)]-s([a,b]))+\mu ({\hat{T}}(s([a,b]))-s(T([a,b])))\\
   		=&\hat{T}([\hat{T}(s(a)),s(b)]+[s(a),\hat{T}(s(b))]+\mu [s(a),s(b)])+\lambda [s(a),s(b)]\\
   		&-s(T([a,T(b)]+[T(a),b]+\mu [a,b]))-\lambda s([a,b])\\
   		=&[T(a),\chi(b)]+[\chi(a),T(b)]+\psi(T(a)\ot T(b))
   	\end{align*}
   	
   	Thus \[ \partial_{\RB}^1(\chi)+\delta^2(\psi)=0.\]
   	
   \end{proof}

	The choice of the section $s$ in fact determines a splitting
	$$\xymatrix{0\ar[r]&  V\ar@<1ex>[r]^{i} &\hat{\frakg}\ar@<1ex>[r]^{p} \ar@<1ex>[l]^{t}& \frakg \ar@<1ex>[l]^{s} \ar[r] & 0}$$
	subject to $t\circ i=\Id_V, t\circ s=0$ and $ it+sp=\Id_{\hat{\frakg}}$.
	Then there is an induced isomorphism of vector spaces
	$$\left(\begin{array}{cc} p& t\end{array}\right): \hat{\frakg}\cong   \frakg\oplus V: \left(\begin{array}{c} s\\ i\end{array}\right).$$
	We now  transfer the extended Rota-Baxter Lie algebra structure on $\hat{\frakg}$ to $\frakg\oplus V$ via this isomorphism. Since the following  proposition is obtained by straightforward  computation, we omit the details for brevity.
	\begin{prop}\label{bracket and operator}
		With the above notation. It is direct to verify that this  endows $\frakg\oplus V$ with a Lie bracket $\pi_\psi$ and an extended Rota-Baxter operator $T_\chi$ defined by
		\begin{align}
			\label{eq:mul}\pi_\psi((a,m),(b,n))&=([a,b],[a,n]-[b,m]+\psi(a,b)),\,\forall a,b\in \frakg,\,m,n\in V,\\
			\label{eq:dif}T_\chi(a,m)&=(T(a),\chi(a)+T_V(m)),\,\forall a\in \frakg,\,m\in V.
		\end{align}
		Moreover, we get an abelian extension
		$$0\to (V, T_V)\stackrel{\left(\begin{array}{cc} 0& 1\end{array}\right) }{\to} (\frakg\oplus V, T_\chi)\stackrel{\left(\begin{array}{c} 1\\ 0\end{array}\right)}{\to} (\frakg, T)\to 0$$
		which is easily seen to be  isomorphic to the original one \eqref{Eq: abelian extension}.
	\end{prop}
	
	By Propositions \ref{psi and chi} and \ref{bracket and operator}, we have:
	\begin{cor}\label{prop:2-cocycle}
		The triple $(\frakg\oplus V,\pi_\psi,T_\chi)$ is an extended Rota-Baxter Lie algebra   if and only if
		$(\psi,\chi)$ is a 2-cocycle  of the extended Rota-Baxter Lie algebra $(\frakg,\pi,T)$ with  coefficients  in $(V,T_V)$.
	\end{cor}
%
%
	\medskip
	
	Now we investigate the influence of different choices of   sections.

	\begin{prop}\ \label{prop: different sections give}
		\begin{itemize}
			\item[(i)] Different choices of the section $s$ give the same  extended Rota-Baxter representation structures on $(V, T_V)$.
			
			\item[(ii)]   The cohomology class of $(\psi,\chi)$ does not depend on the choice of sections.
			
		\end{itemize}
		
	\end{prop}
	\begin{proof}Let $s_1$ and $s_2$ be two distinct sections of $p$.
		We define $\gamma:\frakg\rar V$ by $\gamma(a)=s_1(a)-s_2(a)$.
		
		Since $[u,v]=0$ for all $u,v\in V$,
		$$[s_1(a),m]= [s_2(a),m]+[\gamma(a),m]=[s_2(a),m].$$ So different choices of the section $s$ give the same  extended Rota-Baxter representation structures on $(V, T_V)$;

		We   show that the cohomology class of $(\psi,\chi)$ does not depend on the choice of sections.   We have
		\begin{align*}
			\psi_1(a,b)&=[s_1(a),s_1(b)]-s_1([a,b])\\
			&=[(s_2(a)+\gamma(a)),(s_2(b)+\gamma(b))]-(s_2([a,b])+\gamma([a,b]))\\
			&=([s_2(a),s_2(b)]-s_2([a,b]))+[s_2(a),\gamma(b)]+[\gamma(a),s_2(b)]-\gamma([a,b])\\
			&=([s_2(a),s_2(b)]-s_2([a,b]))+[a,\gamma(b)]+[\gamma(a),b]-\gamma([a,b])\\
			&=\psi_2(a,b)+\partial_{\Lie}(\gamma)(a,b)
		\end{align*}
		and
		\begin{align*}
			\chi_1(a)&={\hat{T}}(s_1(a))-s_1(T(a))\\
			&={\hat{T}}(s_2(a)+\gamma(a))-(s_2(T(a))+\gamma(T(a)))\\
			&=({\hat{T}}(s_2(a))-s_2(T(a)))+{\hat{T}}(\gamma(a))-\gamma(T(a))\\
			&=\chi_2(a)+T_V(\gamma(a))-\gamma(T(a))\\
			&=\chi_2(a)-\delta^1(\gamma)(a).
		\end{align*}
		That is, $(\psi_1,\chi_1)=(\psi_2,\chi_2)+\partial_{\RBA}(\gamma)$. Thus $(\psi_1,\chi_1)$ and $(\psi_2,\chi_2)$ form the same cohomology class  {in $\rmH_{\RBA}^2(\frakg,V)$}.
		
	\end{proof}
	
	We show now the isomorphic abelian extensions give rise to the same cohomology class.
	\begin{prop}Let $V$ be a vector space and  $T_V\in\End_\bfk(V)$. Let $(V, T_V)$ be an extended Rota-Baxter Lie algebra with trivial bracket.
		Let $(\frakg,\pi,T)$ be an extended Rota-Baxter Lie algebra.
		Then any two isomorphic abelian extensions of extended Rota-Baxter Lie algebra $(\frakg, T)$ by  $(V, T_V)$  give rise to the same cohomology class  in $\rmH_{\RBA}^2(\frakg,V)$.
	\end{prop}
	\begin{proof}
		Assume that $(\hat{\frakg}_1,{\hat{T}_1})$ and $(\hat{\frakg}_2,{\hat{T}_2})$ are two isomorphic abelian extensions of $(\frakg,\pi,T)$ by $(V,T_V)$ as is given in \eqref{Eq: isom of abelian extension}. Let $s_1$ be a section of $(\hat{\frakg}_1,{\hat{T}_1})$. As $p_2\circ\zeta=p_1$, we have
		\[p_2\circ(\zeta\circ s_1)=p_1\circ s_1=\Id_{\frakg}.\]
		Therefore, $\zeta\circ s_1$ is a section of $(\hat{\frakg}_2,{\hat{T}_2})$. Denote $s_2:=\zeta\circ s_1$. Since $\zeta$ is a homomorphism of extended Rota-Baxter Lie algebras such that $\zeta|_V=\Id_V$, $\zeta([a,m])=\zeta([s_1(a),m])=[s_2(a),m]=[a,m]$, so $\zeta|_V: V\to V$ is compatible with the induced extended Rota-Baxter representation structures of $V$.
		We have
		\begin{align*}
			\psi_2(a,b)&=[s_2(a),s_2(b)]-s_2([a,b])=[\zeta(s_1(a)),\zeta(s_1(b))]-\zeta(s_1([a,b]))\\
			&=\zeta([s_1(a),s_1(b)]-s_1([a,b]))=\zeta(\psi_1(a,b))\\
			&=\psi_1(a,b)
		\end{align*}
		and
		\begin{align*}
			\chi_2(a)&={\hat{T}_2}(s_2(a))-s_2(T(a))={\hat{T}_2}(\zeta(s_1(a)))-\zeta(s_1(T(a)))\\
			&=\zeta({\hat{T}_1}(s_1(a))-s_1(T(a)))=\zeta(\chi_1(a))\\
			&=\chi_1(a).
		\end{align*}
		Consequently, two isomorphic abelian extensions give rise to the same element in {$\rmH_{\RBA}^2(\frakg,V)$}.
	\end{proof}
	\bigskip

	Now we consider the reverse direction.

	\begin{prop}
		Two cohomologous $2$-cocycles give rise to isomorphic abelian extensions.
	\end{prop}
	\begin{proof}
		
		Given two 2-cocycles $(\psi_1,\chi_1)$ and $(\psi_2,\chi_2)$, we can construct two abelian extensions $(\frakg\oplus V,\cdot_{\psi_1},T_{\chi_1})$ and  $(\frakg\oplus V,\cdot_{\psi_2},T_{\chi_2})$ via Equations~\eqref{eq:mul} and \eqref{eq:dif}. If they represent the same cohomology  class {in $\rmH_{\RBA}^2(\frakg,V)$}, then there exist two linear maps $\gamma_0:\textbf{k}\rightarrow V, \gamma_1:\frakg\to V$ such that $$(\psi_1,\chi_1)=(\psi_2,\chi_2)+(\partial_{\Lie}^1(\gamma_1),-\delta^1(\gamma_1)-\partial_{\RB}^0(\gamma_0)).$$
		Notice that $\partial_{\RB}^0=\delta^1\circ\partial_{\Lie}^0$. Define $\gamma: \frakg\rightarrow V$ to be $\gamma_1+\partial_{\Lie}^0(\gamma_0)$. Then $\gamma$ satisfies
		\[(\psi_1,\chi_1)=(\psi_2,\chi_2)+(\partial_{\Lie}^1(\gamma),-\delta^1(\gamma)).\]
		
		Define $\zeta:\frakg\oplus V\rar \frakg\oplus V$ by
		\[\zeta(a,m):=(a, \gamma(a)+m).\]
		Then $\zeta$ is an isomorphism of these two abelian extensions $(\frakg\oplus V,\cdot_{\psi_1},T_{\chi_1})$ and  $(\frakg\oplus V,\cdot_{\psi_2},T_{\chi_2})$.
	\end{proof}
	
	Finally, by all the propositions above in this section, we have:
	\begin{prop}
		There is a bijection between the isomorphism classes of  abelian extensions of $(\frakg,\pi,T)$ by $(V,T_V)$ and the second cohomology group   ${\rmH}_{\RBA}^2(\frakg,V)$.
	\end{prop}

	\bigskip
	
	\section{Homotopy extended Rota-Baxter Lie algebras}\label{sec 9}
	
	In this section, we will introduce the notion of homotopy extended Rota-Baxter Lie algebras with weight.
	
	Let $V=\bigoplus\limits_{i\in \mathbb{Z}}V^i$ be a graded space.
	\begin{defn}
		 The  \textbf{ graded  Nijenhuis-Richardson bracket}(aka.NR) $[\cdot,\cdot]_{\NR}$ on $\Hom(\rmS(V),V)$ is given as follows:
		for any $P\in \Hom(\rmS^{p+1}(V),V)$ of degree $n$ and $Q\in \Hom(\rmS^{q+1}(V),V)$ of degree m,
		\begin{equation*}
			[P,Q]_{\NR}:=P\bar\circ Q-(-1)^{nm}Q\bar\circ P,
		\end{equation*}
		where $P\bar\circ Q\in \mathrm{Hom}(\rmS^{p+q+1}(V),V)$ of degree $p+q$ is defined by
		\begin{equation*}
			\begin{aligned}
				P\bar\circ Q(v_1,\dots,v_{p+q+1})
				&=\sum_{\sigma\in \Sh(q+1,p)}\varepsilon(\sigma)P(Q(v_{\sigma(1)},\dots, v_{\sigma(q)}),v_{\sigma(q+1)},v_{\sigma(q+2)}\dots, v_{\sigma(p+q+1)}), \end{aligned}
		\end{equation*}
		for homogeneous elements $v_1, \dots, v_{p+q+1}\in V$.
		
	\end{defn}
	
	It is well known that $(\Hom(\rmS(V),V),[\cdot,\cdot]_{\NR})$ is an ordinary graded Lie algebra. Thus we have the following proposition.
	\begin{prop}\label{gla for Linfy}\
		
		\begin{enumerate}
			\item 	Let $\frakg$ be a graded space. Then $(s\mathrm{Hom}(\rmS (\frakg),\frakg),[\cdot,\cdot])$ is a graded Lie algebra(aka. gla) where we define $[sx,sy]:=(-1)^{|x|}s[x,y]_{\bf{NR}}$ for any homogeneous elements $x\in \mathrm{Hom}(\rmS^{p+1}(\frakg),\frakg)$ and $y\in \mathrm{Hom}(\rmS^{q+1}(\frakg),\frakg)$.
			\item An element  $s\pi=\{s\pi_i\}_{i\geqslant 1} \in s\mathrm{Hom}(\rmS(\frakg),\frakg)$ with $\pi_i\in \mathrm{Hom}(\rmS^i(\frakg),\frakg)$  is a Maurer-Cartan element   if and only if $(\frakg,\pi)$ is an  $L_\infty [1]$-algebra.
		\end{enumerate}
	\end{prop}

	Recall $\overline{S}(V)=\bigoplus\limits_{n=1}^\infty \rmS^n(V)$, denote $$\mathfrak{C}_{\Lie}(V)=\Hom(\overline{\rmS}(V),V)$$ and
	$$ \mathfrak{C}_{\RB}(V)=\Hom(\overline{\rmS}(V),V).$$
	Set $$\mathfrak{C}_{\RBA}(V) =s\mathfrak{C}_{\Lie}(V)\oplus\mathfrak{C}_{\RB}(V).$$  Then    $\mathfrak{C}_{\RBA}(V)$ is  an $L_\infty$-algebra. The $L_\infty[1]$ structure is given by:
	
	\begin{thm}\label{thm: Linfinity for homotopy Lie}
		Keeping the above notation and donote $t=\sqrt{\mu^2-4\lambda}$, there exists an $L_\infty[1]$-algebra structure on  ${\mathfrak{C}_{\RBA}}(V)$, where $l_i$ are given by 
		$$l_1(sf)=-A_n f$$
		$$
		l_2(sf,sg)      =    (-1)^{|f|} s[f,g]_{\NR}, $$
		and for $i\geq 2$,
		\begin{equation*}
			\begin{aligned}
				l_i(sf,\xi_1,\cdots,\xi_{i-1})  =&\sum\limits_{\tau\in\Sh(m_1+1,\dots,m_{i-1}+1,n+2-i)}\varepsilon(\tau)\\ &
				-A_{n-i+1}\big(f(\xi_1\otimes\dots\otimes\xi_{i-1}\otimes\Id^{\otimes n+2-i})\big)\tau^{-1}\\
				&+\sum_{k=1}^{i-1}\sum\limits_{\tau\in\Sh(m_1+1,\dots,m_{i-1}+1,n+3-i,m_k)}\varepsilon(\tau)(-1)^{(|f|+\sum\limits_{j=1 }^{k-1}|\xi_j|)|\xi_k|} \\ &
				-B_{n-i+3} \xi_k  \big(f(\xi_1\otimes\dots\otimes\xi_{i-1}\otimes\Id^{\otimes n+3-i})\o \Id^{\otimes m_k}\big)\tau^{-1}	
			\end{aligned}
		\end{equation*}
		
		For homogeneous elements	$f\in\Hom(\rmS^{n+1}(V),V)\subseteq \mathfrak{C}_{\Lie}(V)$, $g\in\Hom(\rmS^{m+1}(V),V)\subseteq \mathfrak{C}_{\Lie}(V)$    	and $\xi_j\in\Hom(\rmS^{m_j+1}(V),V)\subseteq \mathfrak{C}_{\RB}(V)$, $1\leq j\leq i-1$, and all other components vanish.
	\end{thm}
	
	\begin{rmk}
		The proof is  similar to Theorem \ref{thm: Linfinity for extend Lie}.
	\end{rmk}

	\begin{prop-def}\label{Def: homotopy RB algebras}
		Let $V=\bigoplus\limits_{i\in \mathbb{Z}}V^i$ be a graded space. Then a homotopy extended Rota-Baxter Lie algebra $(V,\pi=\{\pi_i\}_{i\geqslant 1}, T=\{T_i\}_{i\geqslant 1})$ of weight $(\mu,\lambda)$, where $\pi_i,T_i\in \Hom(\rmS^i (V),V)$ for each $i$, is defined to be such that $(s\pi,T)$ is a Maurer-Cartan element in the $L_\infty[1]$-algebra ${\mathfrak{C}_{\RBA}}(V)$. The explicit formulas are given as follows.
	\end{prop-def}

	\begin{proof}
		$(s\pi,T)=(\{s\pi_i\}_{i\geqslant 1}, \{T_i\}_{i\geqslant 1}) \in\mathcal{MC}(\mathfrak{C}_{\RBA}(V))$ if and only if
		$[\pi,\pi]_{NR}=0$, i.e. 
		\begin{eqnarray}
			\label{Eq: L infinity}   \sum_{i+j= n,\atop i\geq 0, j\geq 1}\sum_{\tau\in \Sh(j,n-j)} \varepsilon(\tau) \quad\pi_{i+1}\circ(\pi_j\ot \Id^{\ot i})\tau^{-1}=0,
		\end{eqnarray}
		
		and
		\begin{equation*}
			\begin{aligned}
				0 &= 	l_1(s\pi)+\sum_{k=2}^\infty\frac{1}{(k-1)!}l_k(s\pi,\underbrace{T,\dots,T}_{(k-1)\ \mathrm{times}}).
			\end{aligned}
		\end{equation*}

		For arbitrary  $(x_1,x_2,\dots,x_n)\in \rmS^n(V)$, we have
		
			\begin{align}\label{Eq: T infinity}
				0=	&-A_{n-1}\pi_n-\sum\limits_{\tau\in\Sh(j_1,\dots,j_k,m-k)\atop j_1+\dots+j_k+m-k=n}\varepsilon(\tau)\\ &
				A_{m-k-1} \big(\pi_m(T_{j_1}\otimes\dots\otimes T_{j_{k}}\otimes\Id^{\otimes m-k})\big)\tau^{-1}\notag\\
				&-\sum_{k=1}^{m-1}\sum\limits_{\tau\in\Sh(j_2,\dots,j_k,m-k+1,j_1-1)\atop j_1+\dots+j_k+m-k=n}\varepsilon(\tau)\notag\\ &
				B_{m-k+1} T_{j_1}  \big(\pi_m(T_{j_2}\otimes\dots\otimes T_{j_k}\otimes\Id^{\otimes m-k+1})\o \Id^{\otimes j_1-1}\big)\tau^{-1}.\notag
			\end{align}
	
	\end{proof}

	\begin{rmk}
		Notice that the Equation \ref{Eq: L infinity} is exactly the  generalised Jacobi identity \ref{graded Jacobi} and  $(V,\pi)$ is an $L_\infty [1]$-algebra. Then $T=\{T_i\}_{i\geqslant 1}$ is a homotopy extended Rota-Baxter operator on $(V,\pi)$.
	\end{rmk}

	For $n=1,2$, Equation~(\ref{Eq: T infinity})   gives
	\begin{eqnarray}
		\label{operator-differential}   \pi_1\circ T_1=T_1\circ \pi_1,\end{eqnarray}
	and \begin{eqnarray}
		\label{rbo-homotopy}  &  \pi_2\circ(T_1\ot T_1)-T_1\circ \pi_2\circ (\Id\ot T_1)-T_1\circ \pi_2\circ (T_1\ot \Id)-\mu T_1\circ \pi_2-\lambda \pi_2 \\
		\notag   &= -\pi_1\circ T_2+T_2\circ (\Id\ot \pi_1)+T_2\circ(\pi_1\ot \Id).
	\end{eqnarray}
	Equation~(\ref{operator-differential}) implies that $T_1: (V, \pi_1)\to (V, \pi_1)$ is a cochain map,  thus $T_1$ is well-defined on  $\rmH_\bullet(V, \pi_1)$;  Equation~(\ref{rbo-homotopy}) indicates that $T_1$ is an extended Rota-Baxter operator of weight $(\mu,\lambda)$ with respect to $\pi_2$ up to homotopy, whose  obstruction is just the operator $T_2$. As a consequence, $(\rmH_\bullet(V, \pi_1), \pi_2, T_1)$ is  an extended Rota-Baxter Lie algebra.
	
	\bigskip

	\section{Appendix $1$:The proof of Theorem \ref{thm: Linfinity for extend Lie}}\label{proof}

		We just need to justify the equalities 
		\begin{align*}
			\sum_{i=1}^{n+1}\sum_{\sigma\in \Sh(i,n+1-i)}\varepsilon(\sigma)l_{n-i+2}\circ (l_i\otimes \Id^{\otimes n-i+1})\sigma^{-1}(sf,sg,\xi_1,\dots,\xi_{n-1})=0,
		\end{align*}
	   for $|f|=l,|g|=m,|\xi_j|=m_j, j\ge 1.$ 
    	All other cases are trivial.
	
	\begin{align*}
		LHS=&	l_{n+1}(l_1(sf),sg,\xi_1,\dots,\xi_{n-1})+(-1)^{(l+1)(m+1)}l_{n+1}(l_1(sg),sf,\xi_1,\dots,\xi_{n-1})\qquad(\text{1})\\ 
		&+l_{n}(l_2(sf,sg),\xi_1,\dots,\xi_{n-1})\qquad(\text{2})\\  
		&+\sum_{i=2}^{n}\sum_{\sigma\in \Sh(i-1,n-i)} \varepsilon(\sigma)(-1)^{l(m+1)} l_{n-i+2}(sg,l_i(sf,\xi_{\sigma(1)},\dots,\xi_{\sigma(i-1)}),\xi_{\sigma(i)},\dots,\xi_{\sigma(n-1)})\qquad(\text{3})\\ 
		&+\sum_{i=2}^{n}\sum_{\sigma\in \Sh(i-1,n-i)} \varepsilon(\sigma)(-1)^{l+1}l_{n-i+2}(sf,l_i(sg,\xi_{\sigma(1)},\dots,\xi_{\sigma(i-1)}),\xi_{\sigma(i)},\dots,\xi_{\sigma(n-1)})\qquad(\text{4})     
	\end{align*}

	\begin{align*}
		(1)=&-A_{l}(-1)^{l(m+1)}l_{n+1}(sg,\xi^{1'}_0,\xi^{1'}_1,\dots,\xi^{1'}_{n-1})\\
		&\quad(\text{let } \xi^{1'}_{0}:=f \text{ and all other } \xi^{1'}_{k}:=\xi_{k}) \qquad(\text{1.1})\\  
		-&A_{m}(-1)^{l+1}l_{n+1}(sf,\xi^{1''}_{0},\xi^{1''}_1,\dots,\xi^{1''}_{n-1})\\
		& \quad(\text{let } \xi^{1''}_{0}:=g \text{ and all other } \xi^{1''}_{k}:=\xi_{k}) \qquad(\text{1.2})\\   
		=&A_l (-1)^{l(m+1)}\sum\limits_{\tau\in\Sh(m^{1'}_{0}+1,\dots,m^{1'}_{n-1}+1,m+1-n)}sgn(\tau) \\ 
		&A_{m-n}  \big(g(\xi^{1'}_{0}\otimes\dots\otimes\xi^{1'}_{n-1}\otimes\Id^{\otimes m+1-n})\big)\tau^{-1}    \qquad(\text{1.1.1})\\
		&+A_{l}(-1)^{l(m+1)}\sum_{k=1}^{n-1}\sum\limits_{\tau\in\Sh(m^{1'}_{0}+1,\dots,m^{1'}_{n-1}+1,m+2-n,m^{1'}_{k})}sgn(\tau)(-1)^{(l+\sum\limits_{j=1 }^{k-1}m^{1'}_{j})m^{1'}_{k}}  \\ 
		&B_{m-n+2} \xi^{1'}_{k}\big(g(\xi^{1'}_{0}\otimes\dots\otimes\xi^{1'}_{n-1}\otimes\Id^{\otimes m+2-n})\o \Id^{\otimes m^{1'}_{k}}\big)\tau^{-1}  \qquad(\text{1.1.2})\\
		&+A_{m}(-1)^{l+1}\sum\limits_{\tau\in\Sh(m^{1''}_{0}+1,\dots,m^{1''}_{n-1}+1,l+1-n)}sgn(\tau) \\ 
		&A_{l-n} \big(f(\xi^{1''}_{0}\otimes\dots\otimes\xi^{1''}_{n-1}\otimes\Id^{\otimes l+1-n})\big)\tau^{-1}                          \qquad(\text{1.2.1})\\
		&+A_{m}(-1)^{l+1}\sum_{k=1}^{i-1}\sum\limits_{\tau\in\Sh(m^{1''}_{0}+1,\dots,m^{1''}_{n-1}+1,l+2-n,m^{1''}_{k})}sgn(\tau)(-1)^{(m+\sum\limits_{j=1 }^{k-1}m^{1''}_{j})m^{1''}_{k}}  \\ 
		&B_{l-n+2} \xi^{1''}_{k}\big(f(\xi^{1''}_{0}\otimes\dots\otimes\xi^{1''}_{n-1}\otimes\Id^{\otimes l+2-n})\o \Id^{\otimes m^{1''}_{k}}\big)\tau^{-1} \qquad(\text{1.2.2})
	\end{align*}

	\begin{align*}
		(1.1.2)&=A_{l}(-1)^{l(m+1)}\sum_{k=1}^{n-1}\sum\limits_{\tau\in\Sh(m^{1'}_{0}+1,\dots,m^{1'}_{n-1}+1,m+2-n,m^{1'}_{k})}sgn(\tau)(-1)^{(l+\sum\limits_{j=1 }^{k-1}m^{1'}_{j})m^{1'}_{k}}  \\ 
		&B_{m-n+2} \xi^{1'}_{k}\big(g(\xi^{1'}_{0}\otimes\dots\otimes\xi^{1'}_{n-1}\otimes\Id^{\otimes m+2-n})\o \Id^{\otimes m^{1'}_{k}}\big)\tau^{-1}\text{ with $k=0$}\qquad(\text{1.1.2.1})\\
		+& \text{others}\qquad(\text{1.1.2.2}).
	\end{align*}

	\begin{align*}
		(1.2.2)&=A_{m}(-1)^{l+1}\sum_{k=1}^{i-1}\sum\limits_{\tau\in\Sh(m^{1''}_{0}+1,\dots,m^{1''}_{n-1}+1,l+2-n,m^{1''}_{k})}sgn(\tau)(-1)^{(m+\sum\limits_{j=1 }^{k-1}m^{1''}_{j})m^{1''}_{k}}  \\ 
		&B_{l-n+2} \xi^{1''}_{k}\big(f(\xi^{1''}_{0}\otimes\dots\otimes\xi^{1''}_{n-1}\otimes\Id^{\otimes l+2-n})\o \Id^{\otimes m^{1''}_{k}}\big)\tau^{-1} \text{ with $k=0$}\qquad(\text{1.2.2.1})\\
		+& \text{others}\qquad(\text{1.2.2.2}).
	\end{align*}

	\begin{align*}
		(2)=&(-1)^{l} l_{n}( s[f,g]_{\NR},\xi_1,\dots,\xi_{n-1})\\
		=&(-1)^{l+1} 
		\sum\limits_{\tau\in\Sh(m_{1}+1,\dots,m_{n-1}+1,l+m+2-n)}sgn(\tau)  A_{l+m-n+1}\\ 
		&(f\bar{\circ} g)(\xi_{1}\otimes\dots\otimes\xi_{n-1}\otimes\Id^{\otimes m+l+2-n})\tau^{-1} \qquad(\text{2.1.1})\\ 
		&+(-1)^{l+1}
		\sum_{k=1}^{n-1}\sum\limits_{\tau\in\Sh(m_{1}+1,\dots,m_{n-1}+1,m+l+3-n,m_{0})}sgn(\tau) (-1)^{(m+l+\sum\limits_{j=1 }^{k-1}m_{j})m_{k}}  B_{l+m-n+3} \\ 
		&\xi_{k}\big((f\bar{\circ} g)(\xi_{1}\otimes\dots\otimes\xi_{n-1}\otimes\Id^{\otimes m+l+3-n})\o \Id^{\otimes m_{k}}\big)\tau^{-1}    \qquad(\text{2.1.2})\\	
		&+(-1)^{l+ml}
		\sum\limits_{\tau\in\Sh(m_{1}+1,\dots,m_{n-1}+1,l+m+2-n)}sgn(\tau) A_{l+m-n+1}\\ 
		&(g\bar{\circ} f)(\xi_{1}\otimes\dots\otimes\xi_{n-1}\otimes\Id^{\otimes m+l+2-n})\tau^{-1} \qquad(\text{2.2.1})\\
		&+(-1)^{l+ml}
		\sum_{k=1}^{n-1}\sum\limits_{\tau\in\Sh(m_{1}+1,\dots,m_{n-1}+1,m+l+3-n,m_{k})}sgn(\tau)(-1)^{(m+l+\sum\limits_{j=1 }^{k-1}m_{j})m_{k}} B_{l+m-n+3} \\ 
		&\xi_{k}\big((g\bar{\circ} f)(\xi_{1}\otimes\dots\otimes\xi_{n-1}\otimes\Id^{\otimes m+l+3-n})\o \Id^{\otimes m_{k}}\big)\tau^{-1} \qquad(\text{2.2.2})
	\end{align*}

	\begin{align*}
		(3)=&\sum_{i=2}^{n}\sum_{\sigma\in \Sh(i-1,n-i)} \varepsilon(\sigma)(-1)^{l(m+1)}	l_{n-i+2}(sg, \sum\limits_{\tau\in\Sh(m_{\sigma(1)}+1,\dots,m_{\sigma(i-1)}+1,l+2-i)}sgn(\tau)  \\ 
		&-A_{l-i+1} \xi^{3'}_{\sigma(i-1)},\xi^{3'}_{\sigma(i)},\dots,\xi^{3'}_{\sigma(n-1)})\qquad(\text{3.1})\\     
		&(\text{Let}\quad \xi^{3'}_{\sigma(i-1)}:=\big(f(\xi_{\sigma(1)}\otimes\dots\otimes\xi_{\sigma(i-1)}\otimes\Id^{\otimes l+2-i})\big)\tau^{-1}\text{ and all other } \xi^{3'}_{k}:=\xi_{k})\\
		&+\sum_{i=2}^{n}\sum_{\sigma\in \Sh(i-1,n-i)} \varepsilon(\sigma)(-1)^{l(m+1)}	l_{n-i+2}(sg, \sum_{k=1}^{i-1}\sum\limits_{\tau\in\Sh(m_{\sigma(1)}+1,\dots,m_{\sigma(i-1)}+1,l+3-i,m_{\sigma(k)})}sgn(\tau)(-1)^{(l+\sum\limits_{j=1 }^{k-1}m_{\sigma(j)})m_{\sigma(k)}}  \\ 
		&-B_{l-i+3} \xi^{3''}_{\sigma(i-1)},\xi^{3''}_{\sigma(i)},\dots,\xi^{3''}_{\sigma(n-1)}) \qquad(\text{3.2})\\  
		&(\text{Let}\quad \xi^{3''}_{\sigma(i-1)}:= \xi_{\sigma(k)}\big(f(\xi_{\sigma(1)}\otimes\dots\otimes\xi_{\sigma(i-1)}\otimes\Id^{\otimes l+3-i})\o \Id^{\otimes m_{\sigma(k)}}\big)\tau^{-1}\text{ and all other } \xi^{3''}_{k}:=\xi_{k})\\
		=&\sum_{i=2}^{n}\sum_{\sigma\in \Sh(i-1,n-i)} \varepsilon(\sigma)(-1)^{l(m+1)} \sum\limits_{\tau\in\Sh(m_{\sigma(1)}+1,\dots,m_{\sigma(i-1)}+1,l+2-i)}sgn(\tau) \sum\limits_{\tau_1\in\Sh(m^{3'}_{\sigma(i-1)}+1,\dots,m^{3'}_{\sigma(n-1)}+1,m-n-i)}sgn(\tau_1) \\ 
		&A_{l-i+1} A_{m-n+i-1} \big(g(\xi^{3'}_{\sigma(i-1)}\otimes\dots\otimes\xi^{3'}_{\sigma(n-1)}\otimes\Id^{\otimes m-n+i})\big)\tau_{1}^{-1}                   \qquad(\text{3.1.1})\\
		&+\sum_{i=2}^{n}\sum_{\sigma\in \Sh(i-1,n-i)} \varepsilon(\sigma)(-1)^{l(m+1)} \sum\limits_{\tau\in\Sh(m_{\sigma(1)}+1,\dots,m_{\sigma(i-1)}+1,l+2-i)} \sum\limits_{\tau_1\in\Sh(m^{3'}_{\sigma(i-1)}+1,\dots,m^{3'}_{\sigma(n-1)}+1,m-n-i+1,m^{3'}_{\sigma(k)})}\\
		&\sum_{k=i-1}^{n-1}sgn(\tau)sgn(\tau_1) (-1)^{(m+\sum\limits_{j=i-1 }^{k-1}m^{3'}_{\sigma(j)})m^{3'}_{\sigma(k)}}A_{l-i+1} B_{m-n+i+1}\\
		& \xi^{3'}_{\sigma(k)}\big(g(\xi^{3'}_{\sigma(i-1)}\otimes\dots\otimes\xi^{3'}_{\sigma(n-1)}\otimes\Id^{\otimes m-n+i+1})\o \id^{\o m^{3'}_{\sigma(k)}}\big)\tau_{1}^{-1}            \qquad(\text{3.1.2})\\
		&+\sum_{i=2}^{n}\sum_{\sigma\in \Sh(i-1,n-i)} \varepsilon(\sigma)(-1)^{l(m+1)} \sum\limits_{\tau\in\Sh(m_{\sigma(1)}+1,\dots,m_{\sigma(i-1)}+1,l+3-i,m_{\sigma(k)})} \sum\limits_{\tau_1\in\Sh(m^{3''}_{\sigma(i-1)}+1,\dots,m^{3''}_{\sigma(n-1)}+1,m-n-i)}\\
		&\sum_{k=1}^{i-1}sgn(\tau)sgn(\tau_1)(-1)^{(l+\sum\limits_{j=1 }^{k-1}m_{\sigma(j)})m_{\sigma(k)}} B_{l-i+3} A_{m-n+i-1}\\
		& \big(g(\xi^{3''}_{\sigma(i-1)}\otimes\dots\otimes\xi^{3''}_{\sigma(n-1)}\otimes\Id^{\otimes m-n+i})\big)\tau_{1}^{-1}                 \qquad(\text{3.2.1})\\
		&+\sum_{i=2}^{n}\sum_{\sigma\in \Sh(i-1,n-i)} \varepsilon(\sigma)(-1)^{l(m+1)} \sum\limits_{\tau\in\Sh(m_{\sigma(1)}+1,\dots,m_{\sigma(i-1)}+1,l+3-i,m_{\sigma(k)})} \sum\limits_{\tau_1\in\Sh(m^{3''}_{\sigma(i-1)}+1,\dots,m^{3''}_{\sigma(n-1)}+1,m-n-i+1,m^{3''}_{\sigma(k')})}\\
		&\sum_{k=1}^{i-1}\sum_{k'=i-1}^{n-1}sgn(\tau)sgn(\tau_1)(-1)^{(l+\sum\limits_{j=1 }^{k-1}m_{\sigma(j)})m_{\sigma(k)}+(m+\sum\limits_{j=i-1 }^{k'-1}m^{3''}_{\sigma(j)})m^{3''}_{\sigma(k')}} B_{l-i+3} B_{m-n+i+1}\\ &\xi^{3''}_{\sigma(k')}\big(g(\xi^{3''}_{\sigma(i-1)}\otimes\dots\otimes\xi^{3''}_{\sigma(n-1)}\otimes\Id^{\otimes m-n+i+1})\o \id^{\o m^{3''}_{\sigma(k')}}\big)\tau_{1}^{-1} \qquad(\text{3.2.2})
	\end{align*}

	\begin{align*}
		(3.1.2)&=\sum_{i=2}^{n}\sum_{\sigma\in \Sh(i-1,n-i)} \varepsilon(\sigma)(-1)^{l(m+1)} \sum\limits_{\tau\in\Sh(m_{\sigma(1)}+1,\dots,m_{\sigma(i-1)}+1,l+2-i)} \sum\limits_{\tau_1\in\Sh(m^{3'}_{\sigma(i-1)}+1,\dots,m^{3'}_{\sigma(n-1)}+1,m-n-i+1,m^{3'}_{\sigma(k)})}\\
		&\sum_{k=i-1}^{n-1}sgn(\tau)sgn(\tau_1) (-1)^{(m+\sum\limits_{j=i-1 }^{k-1}m^{3'}_{\sigma(j)})m^{3'}_{\sigma(k)}}A_{l-i+1} B_{m-n+i+1}\\
		& \xi^{3'}_{\sigma(k)}\big(g(\xi^{3'}_{\sigma(i-1)}\otimes\dots\otimes\xi^{3'}_{\sigma(n-1)}\otimes\Id^{\otimes m-n+i+1})\o \id^{\o m^{3'}_{\sigma(k)}}\big)\tau_{1}^{-1} \\ 
		&\text{ with $\sigma(k)\neq i-1$}\qquad(\text{3.1.2.1})\\
		+&\sum_{i=2}^{n}\sum_{\sigma\in \Sh(i-1,n-i)} \varepsilon(\sigma)(-1)^{l(m+1)} \sum\limits_{\tau\in\Sh(m_{\sigma(1)}+1,\dots,m_{\sigma(i-1)}+1,l+2-i)} \sum\limits_{\tau_1\in\Sh(m^{3'}_{\sigma(i-1)}+1,\dots,m^{3'}_{\sigma(n-1)}+1,m-n-i+1,m^{3'}_{\sigma(k)})}\\
		&\sum_{k=i-1}^{n-1}sgn(\tau)sgn(\tau_1) (-1)^{(m+\sum\limits_{j=i-1 }^{k-1}m^{3'}_{\sigma(j)})m^{3'}_{\sigma(k)}}A_{l-i+1} B_{m-n+i+1}\\
		& \xi^{3'}_{\sigma(k)}\big(g(\xi^{3'}_{\sigma(i-1)}\otimes\dots\otimes\xi^{3'}_{\sigma(n-1)}\otimes\Id^{\otimes m-n+i+1})\o \id^{\o m^{3'}_{\sigma(k)}}\big)\tau_{1}^{-1} \\ 
		&\text{ with $\sigma(k)=i-1$,  $f$ and $g$ are directly composed}\qquad(\text{3.1.2.2})\\
		+& \text{others}\qquad(\text{3.1.2.3}).
	\end{align*}
	
	\begin{rmk}
     $f$ and $g$ are directly composed means that $f\circ (g\otimes \Id^{\otimes n})$ appears  in this iterative composition.
	\end{rmk}

	\begin{align*}
		(3.2.2)&=\sum_{i=2}^{n}\sum_{\sigma\in \Sh(i-1,n-i)} \varepsilon(\sigma)(-1)^{l(m+1)} \sum\limits_{\tau\in\Sh(m_{\sigma(1)}+1,\dots,m_{\sigma(i-1)}+1,l+3-i,m_{\sigma(k)})} \sum\limits_{\tau_1\in\Sh(m^{3''}_{\sigma(i-1)}+1,\dots,m^{3''}_{\sigma(n-1)}+1,m-n-i+1,m^{3''}_{\sigma(k')})}\\
		&\sum_{k=1}^{i-1}\sum_{k'=i-1}^{n-1}sgn(\tau)sgn(\tau_1)(-1)^{(l+\sum\limits_{j=1 }^{k-1}m_{\sigma(j)})m_{\sigma(k)}+(m+\sum\limits_{j=i-1 }^{k'-1}m^{3''}_{\sigma(j)})m^{3''}_{\sigma(k')}} B_{l-i+3} B_{m-n+i+1}\\ &\xi^{3''}_{\sigma(k')}\big(g(\xi^{3''}_{\sigma(i-1)}\otimes\dots\otimes\xi^{3''}_{\sigma(n-1)}\otimes\Id^{\otimes m-n+i+1})\o \id^{\o m^{3''}_{\sigma(k')}}\big)\tau_{1}^{-1}\\
		&\text{ with $\sigma(k')=i-1$,  $f$ and $g$ are directly composed}\qquad(\text{3.2.2.1})\\
		+& \text{others}\qquad(\text{3.2.2.2}).
	\end{align*}
	
	\begin{align*}
		(4)=&\sum_{i=2}^{n}\sum_{\sigma\in \Sh(i-1,n-i)} \varepsilon(\sigma)(-1)^{l+1}l_{n-i+2}(sf,\sum\limits_{\tau\in\Sh(m_{\sigma(1)}+1,\dots,m_{\sigma(i-1)}+1,m+2-i)}sgn(\tau) \\ 
		&-A_{m-i+1}  \xi^{4'}_{\sigma(i-1)},\xi^{4'}_{\sigma(i)},\dots,\xi^{4'}_{\sigma(n-1)})\qquad(\text{4.1})\\     
		&(\text{Let}\quad \xi^{4'}_{\sigma(i-1)}:=\big(g(\xi_{\sigma(1)}\otimes\dots\otimes\xi_{\sigma(i-1)}\otimes\Id^{\otimes m+2-i})\big)\tau^{-1} \text{ and all other } \xi^{4'}_{k}:=\xi_{k})\\
		&+\sum_{i=2}^{n} \sum_{\sigma\in \Sh(i-1,n-i)} \varepsilon(\sigma)(-1)^{l+1} l_{n-i+2}(sf,\sum_{k=1}^{i-1}\sum\limits_{\tau\in\Sh(m_{\sigma(1)}+1,\dots,m_{\sigma(i-1)}+1,m+3-i,m_{\sigma(k)})}sgn(\tau)(-1)^{(m+\sum\limits_{j=1 }^{k-1}m_{\sigma(j)})m_{\sigma(k)}}  \\ 
		&-B_{m-i+3}\xi_{\sigma(i-1)},\xi^{4''}_{\sigma(i)},\dots,\xi^{4''}_{\sigma(n-1)})    \qquad(\text{4.2})\\  
		&(\text{Let}\quad \xi^{4''}_{\sigma(i-1)}:= \xi_{\sigma(k)}\big(g(\xi_{\sigma(1)}\otimes\dots\otimes\xi_{\sigma(i-1)}\otimes\Id^{\otimes m+3-i})\o \Id^{\otimes m_{\sigma(k)}}\big)\tau^{-1}\text{ and all other } \xi^{4''}_{k}:=\xi_{k})\\
		=&\sum_{i=2}^{n}\sum_{\sigma\in \Sh(i-1,n-i)} \varepsilon(\sigma)(-1)^{l+1} \sum\limits_{\tau\in\Sh(m_{\sigma(1)}+1,\dots,m_{\sigma(i-1)}+1,m+2-i)}sgn(\tau) \sum\limits_{\tau_1\in\Sh(m^{4'}_{\sigma(i-1)}+1,\dots,m^{4'}_{\sigma(n-1)}+1,l-n-i)}sgn(\tau_1) \\ 
		&A_{m-i+1} A_{l-n+i-1} \big(f(\xi^{4'}_{\sigma(i-1)}\otimes\dots\otimes\xi^{4'}_{\sigma(n-1)}\otimes\Id^{\otimes m-n+i})\big)\tau_{1}^{-1}            \qquad(\text{4.1.1})\\
		&+\sum_{i=2}^{n}\sum_{\sigma\in \Sh(i-1,n-i)} \varepsilon(\sigma)(-1)^{l+1}\sum_{k=i-1}^{n-1} \sum\limits_{\tau\in\Sh(m_{\sigma(1)}+1,\dots,m_{\sigma(i-1)}+1,m+2-i)} \sum\limits_{\tau_1\in\Sh(m^{4'}_{\sigma(i-1)}+1,\dots,m^{4'}_{\sigma(n-1)}+1,l-n-i+1,m^{4'}_{\sigma(k)})}\\
		&sgn(\tau)sgn(\tau_1)(-1)^{(l+\sum\limits_{j=i-1 }^{k-1}m^{4'}_{\sigma(j)})m^{4'}_{\sigma(k)}}A_{m-i+1} B_{m-n+i+1}\\ &\xi^{4'}_{\sigma(k)}\big(f(\xi^{4'}_{\sigma(i-1)}\otimes\dots\otimes\xi^{4'}_{\sigma(n-1)}\otimes\Id^{\otimes l-n+i})\o \id^{\o m^{4'}_{\sigma(k)}}\big)\tau_{1}^{-1}      \qquad(\text{4.1.2})\\
		&+\sum_{i=2}^{n}\sum_{\sigma\in \Sh(i-1,n-i)} \varepsilon(\sigma)(-1)^{l+1}\sum_{k=1}^{i-1} \sum\limits_{\tau\in\Sh(m_{\sigma(1)}+1,\dots,m_{\sigma(i-1)}+1,m+3-i,m_{\sigma(k)})}sgn(\tau) \sum\limits_{\tau_1\in\Sh(m^{4''}_{\sigma(i-1)}+1,\dots,m^{4''}_{\sigma(n-1)}+1,l-n-i)} \\ 
		&sgn(\tau_1)(-1)^{(m+\sum\limits_{j=1 }^{k-1}m_{\sigma(j)})m_{\sigma(k)}}  B_{m-i+3} A_{l-n+i-1} \big(f(\xi^{4''}_{\sigma(i-1)}\otimes\dots\otimes\xi^{4''}_{\sigma(n-1)}\otimes\Id^{\otimes l-n+i})\big)\tau_{1}^{-1}             \qquad(\text{4.2.1})\\
		&+\sum_{i=2}^{n}\sum_{\sigma\in \Sh(i-1,n-i)} \varepsilon(\sigma)(-1)^{l+1}\sum_{k=1}^{i-1}\sum_{k'=i-1}^{n-1} \sum\limits_{\tau\in\Sh(m_{\sigma(1)}+1,\dots,m_{\sigma(i-1)}+1,m+3-i,m_{\sigma(k)})} \sum\limits_{\tau_1\in\Sh(m^{4''}_{\sigma(i-1)}+1,\dots,m^{4''}_{\sigma(n-1)}+1,l-n-i+1,m^{4''}_{\sigma(k')})}\\
		&sgn(\tau)sgn(\tau_1) (-1)^{(m+\sum\limits_{j=1 }^{k-1}m_{\sigma(j)})m_{\sigma(k)}+(l+\sum\limits_{j=i-1 }^{k'-1}m^{4''}_{\sigma(j)})m^{4''}_{\sigma(k')}}B_{m-i+3} B_{l-n+i+1}\\ &\xi^{4''}_{\sigma(k')}\big(f(\xi^{4''}_{\sigma(i-1)}\otimes\dots\otimes\xi^{4''}_{\sigma(n-1)}\otimes\Id^{\otimes l-n+i})\o \id^{\o m_{\sigma(k')}}\big)\tau_{1}^{-1} \qquad(\text{4.2.2})
	\end{align*}

	\begin{equation*}
		\begin{aligned}
			(4.1.2)&=\sum_{i=2}^{n}\sum_{\sigma\in \Sh(i-1,n-i)} \varepsilon(\sigma)(-1)^{l+1}\sum_{k=i-1}^{n-1} \sum\limits_{\tau\in\Sh(m_{\sigma(1)}+1,\dots,m_{\sigma(i-1)}+1,m+2-i)} \sum\limits_{\tau_1\in\Sh(m^{4'}_{\sigma(i-1)}+1,\dots,m^{4'}_{\sigma(n-1)}+1,l-n-i+1,m^{4'}_{\sigma(k)})}\\
			&sgn(\tau)sgn(\tau_1)(-1)^{(l+\sum\limits_{j=i-1 }^{k-1}m^{4'}_{\sigma(j)})m^{4'}_{\sigma(k)}}A_{m-i+1} B_{m-n+i+1}\\ &\xi^{4'}_{\sigma(k)}\big(f(\xi^{4'}_{\sigma(i-1)}\otimes\dots\otimes\xi^{4'}_{\sigma(n-1)}\otimes\Id^{\otimes l-n+i})\o \id^{\o m^{4'}_{\sigma(k)}}\big)\tau_{1}^{-1}\\
			&\text{ with $\sigma(k)\neq i-1$}\qquad(\text{4.1.2.1})\\
			+&\sum_{i=2}^{n}\sum_{\sigma\in \Sh(i-1,n-i)} \varepsilon(\sigma)(-1)^{l+1}\sum_{k=i-1}^{n-1} \sum\limits_{\tau\in\Sh(m_{\sigma(1)}+1,\dots,m_{\sigma(i-1)}+1,m+2-i)} \sum\limits_{\tau_1\in\Sh(m^{4'}_{\sigma(i-1)}+1,\dots,m^{4'}_{\sigma(n-1)}+1,l-n-i+1,m^{4'}_{\sigma(k)})}\\
			&sgn(\tau)sgn(\tau_1)(-1)^{(l+\sum\limits_{j=i-1 }^{k-1}m^{4'}_{\sigma(j)})m^{4'}_{\sigma(k)}}A_{m-i+1} B_{m-n+i+1}\\ &\xi^{4'}_{\sigma(k)}\big(f(\xi^{4'}_{\sigma(i-1)}\otimes\dots\otimes\xi^{4'}_{\sigma(n-1)}\otimes\Id^{\otimes l-n+i})\o \id^{\o m^{4'}_{\sigma(k)}}\big)\tau_{1}^{-1}\\
			&\text{ with $\sigma(k)= i-1$, $g$ and $f$ are directly composed}\qquad(\text{4.1.2.2})\\
			+& \text{others}\qquad(\text{4.1.2.3}).
		\end{aligned}
	\end{equation*}

	\begin{equation*}
		\begin{aligned}
			(4.2.2)=&\sum_{i=2}^{n}\sum_{\sigma\in \Sh(i-1,n-i)} \varepsilon(\sigma)(-1)^{l+1}\sum_{k=1}^{i-1}\sum_{k'=i-1}^{n-1} \sum\limits_{\tau\in\Sh(m_{\sigma(1)}+1,\dots,m_{\sigma(i-1)}+1,m+3-i,m_{\sigma(k)})}\\ &\sum\limits_{\tau_1\in\Sh(m^{4''}_{\sigma(i-1)}+1,\dots,m^{4''}_{\sigma(n-1)}+1,l-n-i+1,m^{4''}_{\sigma(k')})}sgn(\tau)sgn(\tau_1) (-1)^{(m+\sum\limits_{j=1 }^{k-1}m_{\sigma(j)})m_{\sigma(k)}+(l+\sum\limits_{j=i-1 }^{k'-1}m^{4''}_{\sigma(j)})m^{4''}_{\sigma(k')}}B_{m-i+3} B_{l-n+i+1}\\ &\xi^{4''}_{\sigma(k')}\big(f(\xi^{4''}_{\sigma(i-1)}\otimes\dots\otimes\xi^{4''}_{\sigma(n-1)}\otimes\Id^{\otimes l-n+i})\o \id^{\o m_{\sigma(k')}}\big)\tau_{1}^{-1}\\
			&\text{ with $\sigma(k')= i-1$, $g$ and $f$ are directly composed}\qquad(\text{4.2.2.1})\\
			+& \text{others}\qquad(\text{4.2.2.2}).
		\end{aligned}
	\end{equation*}
	
	With the same form of operation of iterative composition, using the Koszul sign and Lemma \ref{key lemma}, we have the following identities,
	\begin{align*}
		&(1.1.2.1)+(1.2.1)+(2.1.1)+(3.1.2.2)+(4.1.1)=0,\\
		&(1.2.2.2)+(2.1.2)+(3.2.2.1)+(4.1.2.1)=0,\\
		&(1.1.1)+(1.2.2.1)+(2.2.1)+(3.1.1)+(4.1.2.2)=0,\\
		&(1.1.2.2)+(2.2.2)+(3.1.2.1)+(4.2.2.1)=0,\\
		&(3.1.2.3)+(4.2.1)=0,\\
		&(3.2.1)+(4.1.2.3)=0,\\
		&(3.2.2.2)+(4.2.2.2)=0.
	\end{align*}

	For example $(1.1.2.1)+(1.2.1)+(2.1.1)+(3.1.2.2)+(4.1.1)=0$.
	
	\begin{align*}
		LHS=&(-1)^{l}\sum_{k=1}^{n-1}\sum\limits_{\tau\in\Sh(m^{1'}_{0}+1,\dots,m^{1'}_{n-1}+1,m+2-n,m^{1'}_{k})}sgn(\tau)(-1)^{ml+(l+\sum\limits_{j=1 }^{k-1}m^{1'}_{j})m^{1'}_{k}}  \\ 
		&A_{l}B_{m-n+2} \xi^{1'}_{k}\big(g(\xi^{1'}_{0}\otimes\dots\otimes\xi^{1'}_{n-1}\otimes\Id^{\otimes m+2-n})\o \Id^{\otimes m^{1'}_{k}}\big)\tau^{-1}\text{ with $k=0$}\\
		+&(-1)^{l}\sum_{i=2}^{n}\sum_{\sigma\in \Sh(i-1,n-i)} \varepsilon(\sigma) \sum\limits_{\tau\in\Sh(m_{\sigma(1)}+1,\dots,m_{\sigma(i-1)}+1,l+2-i)} \sum\limits_{\tau_1\in\Sh(m^{3'}_{\sigma(i-1)}+1,\dots,m^{3'}_{\sigma(n-1)}+1,m-n-i+1,m^{3'}_{\sigma(k)})}\\
		&\sum_{k=i-1}^{n-1}sgn(\tau)sgn(\tau_1) (-1)^{ml+(m+\sum\limits_{j=i-1 }^{k-1}m^{3'}_{\sigma(j)})m^{3'}_{\sigma(k)}}A_{l-i+1} B_{m-n+i+1}\\
		& \xi^{3'}_{\sigma(k)}\big(g(\xi^{3'}_{\sigma(i-1)}\otimes\dots\otimes\xi^{3'}_{\sigma(n-1)}\otimes\Id^{\otimes m-n+i+1})\o \id^{\o m^{3'}_{\sigma(k)}}\big)\tau_{1}^{-1} \\ 
		&\text{ with $\sigma(k)=i-1$ and  composite of $f$ and $g$ directly}\\
		+& (-1)^{l+1}
		\sum\limits_{\tau\in\Sh(m_{1}+1,\dots,m_{n-1}+1,l+m+2-n)}sgn(\tau)A_{l+m-n+1}\\ 
		&(f\bar{\circ} g)(\xi_{1}\otimes\dots\otimes\xi_{n-1}\otimes\Id^{\otimes m+l+2-n})\tau^{-1}\\ 
		&+(-1)^{l+1}\sum_{i=2}^{n}\sum_{\sigma\in \Sh(i-1,n-i)} \varepsilon(\sigma) \sum\limits_{\tau\in\Sh(m_{\sigma(1)}+1,\dots,m_{\sigma(i-1)}+1,m+2-i)}sgn(\tau) \sum\limits_{\tau_1\in\Sh(m^{4'}_{\sigma(i-1)}+1,\dots,m^{4'}_{\sigma(n-1)}+1,l-n-i)}sgn(\tau_1) \\ 
		&A_{m-i+1} A_{l-n+i-1} \big(f(\xi^{4'}_{\sigma(i-1)}\otimes\dots\otimes\xi^{4'}_{\sigma(n-1)}\otimes\Id^{\otimes m-n+i})\big)\tau_{1}^{-1} \\
		+&(-1)^{l+1} \sum\limits_{\tau\in\Sh(m^{1''}_{0}+1,\dots,m^{1''}_{n-1}+1,l+1-n)}sgn(\tau) \\ 
		&A_{m}A_{l-n} \big(f(\xi^{1''}_{0}\otimes\dots\otimes\xi^{1''}_{n-1}\otimes\Id^{\otimes l+1-n})\big)\tau^{-1}   \\
    	=&0 \text{ by Lemma \ref{key lemma}}.
	\end{align*}
	
	\bigskip

     \section{Appendix $2$: The controlling algebra of Rota-Baxter Lie algebras with weight} \label{Linfty for weight} 
     
     For Rota-Baxter Lie algebras with weight, we have a useful technique to get the controlling algebra. Hence we provide another method to obtain the $L_\infty [1]$-algebra, which coincides with the one in Corollary \ref{wRBLie}.
     
       \subsection{A generalised  version of derived bracket technique}\ \label{Subsect: derived bracket technique}

      The technique of derived brackets is invented by   Voronov   \cite{Vor05}. Now we introduce a generalised version of the derived bracket technique since we need a modified $L_\infty$-algebra related with weight $\mu\in \bfk$.

      \begin{defn}
      	A \textbf{generalised V-datum} is a data  $(\frakL, \frakm,\iota_\frakm, \fraka, \iota_\fraka, P,\Delta)$, where
      	\begin{itemize}
      		\item[$(1)$] $(\frakL,[-,-])$ is an ordinary graded Lie algebra,
      		\item[$(2)$] $(\frakm, [-,-]_\frakm)$ is an ordinary graded Lie  algebra together with an injective linear  map $\iota_\frakm: \frakm\to \frakL$ which is a    homomorphism of  ordinary graded Lie algebras,
      		\item[$(3)$] $\fraka$ is an abelian graded Lie  algebra equipped  with an injective linear  map $\iota_\fraka: \fraka\to \frakL$ which is a    homomorphism of ordinary graded Lie algebras,
      		\item[$(4)$] $P:\frakL\to \fraka$ is a linear map such that   $P\circ \iota_\fraka=\Id_\fraka$ and $\mathrm{Ker}(P)$ is an ordinary graded Lie subalgebra of $\frakL$,
      		\item[$(5)$] $\Delta\in \mathrm{Ker}(P)^1$ satisfying $[\Delta,\Delta]=0$ and  $[\Delta,\iota_\frakm(\frakm)]\subset \iota_\frakm(\frakm)$.
      	\end{itemize}
      \end{defn}

      \begin{thm}\label{theo:V-data}
      	Let $(\frakL, \frakm,\iota_\frakm, \fraka, \iota_\fraka, P,\Delta)$ be a generalised V-datum.
      	
      	Then  the graded vector space
      	$s \frakL\oplus \fraka$  has  an $L_\infty[1]$-algebra structure  which is defined  as follows$\colon$
      	$$\begin{array}{rcl}
      		l_1(sf)     &  =    & (-s [\Delta,f], P(f)),\\
      		l_1(\xi)     &  =    & P[\Delta,\iota_{\fraka}(\xi)],\\
      		l_2(sf,sg)   &  =    &(-1)^{|f|} s [f,g] , \\
      		l_i(sf,\xi_1,\cdots,\xi_{i-1})   &  =    & P[\cdots[f,\iota_{\fraka}(\xi_1)] ,\cdots,\iota_{\fraka}(\xi_{i-1})] ,  i\geq2,\\
      		l_i(\xi_1,\cdots,\xi_{i})   &  =    & P[\cdots[\Delta,\iota_{\fraka}(\xi_1)] ,\cdots,\iota_{\fraka}(\xi_{i})] ,  i\geq2,
      	\end{array}$$
      	for homogeneous elements $f,g\in\frakL$, $\xi,\xi_1,\cdots,\xi_i\in\fraka$ and all other components  of   $\{l_i\}_{i=1}^{+\infty}$ vanish.
      	
      	Similarly, the graded vector space
      	$s \frakm\oplus \fraka$  has  also an $L_\infty[1]$-algebra structure  which is defined  as follows$\colon$
      	$$\begin{array}{rcl}
      		l_1(sf)     &  =    & (-s \iota^{-1}_\frakm[\Delta,\iota_\frakm(f)], P\iota_\frakm(f)),\\
      		l_1(\xi)     &  =    & P[\Delta,\iota_{\fraka}(\xi)],\\
      		l_2(sf,sg)   &  =    &(-1)^{|f|} s [f,g]_\frakm , \\
      		l_i(sf,\xi_1,\cdots,\xi_{i-1})   &  =    & P[\cdots[\iota_\frakm(f),\iota_{\fraka}(\xi_1)] ,\cdots,\iota_{\fraka}(\xi_{i-1})] ,  i\geq2,\\
      		l_i(\xi_1,\cdots,\xi_{i})   &  =    & P[\cdots[\Delta,\iota_{\fraka}(\xi_1)] ,\cdots,\iota_{\fraka}(\xi_{i})] ,  i\geq2,
      	\end{array}$$
      	for homogeneous elements $f,g\in\frakm$, $\xi,\xi_1,\cdots,\xi_i\in\fraka$ and all other components of $\{l_i\}_{i=1}^{+\infty}$  vanish.

      	Moreover, there exists an injective  $L_\infty[1]$-algebra homomorphism $\iota:s \frak{m}\oplus \fraka\to s \frakL\oplus \fraka$ induced by $\iota_\frakm$.
      \end{thm}

      \begin{defn} Let $(\frakL, \frakm,\iota_\frakm, \fraka, \iota_\fraka, P,\Delta)$ and  $(\mathfrak{L'}, \mathfrak{m'},\iota_{\frakm'}, \mathfrak{a'}, \iota_{\fraka'},P',\Delta')$ be two   generalised  V-data.
      	A \textbf{morphism} between them  is  a triple $f=(f_\frakL, f_\frakm, f_\fraka)$,  where $f_\frakL:\frakL\to\mathfrak{L'}$,
      	$f_\frakm: \frakm\to\mathfrak{m'}$ and $f_\fraka: \fraka\to\mathfrak{a'}$  are three  homomorphisms of graded Lie algebras such that
      	$f_\frakL\circ \iota_\frakm= \iota_{\frakm'}\circ f_\frakm$, $f_\frakL\circ \iota_\fraka= \iota_{\fraka'}\circ f_\fraka$,
      	$f_\fraka\circ P=P'\circ f_\frakL$ and $f_\frakL(\Delta)=\Delta'$.
      \end{defn}

      The following   result is obvious.
      \begin{prop}\label{prop:L_infty morphism}
      	Given a morphism of generalised V-data $$f: (\frakL, \frakm,\iota_\frakm, \fraka, \iota_\fraka, P,\Delta)\to (\mathfrak{L'}, \mathfrak{m'},\iota_{\frakm'}, \mathfrak{a'}, \iota_{\fraka'},P',\Delta'),$$  there exists an $L_\infty[1]$-algebra homomorphism $$\tilde{f}: s\frakm\oplus\fraka \to s\mathfrak{m'}\oplus\mathfrak{a'}$$ induced by $f$.
      \end{prop}
  
      \begin{rmk}
      	The generalised V datum  is a refinement of the original V datum. The results above follow from the theory of original V data. Note that the original V datum $(\frakL,  \fraka,  P,\Delta)$ is a generalised V datum $(\frakL, \frakL,\iota_\frakL=\Id, \fraka, \iota_\fraka, P,\Delta) $ and $s \frakL\oplus \fraka$ is an  $L_\infty[1]$-algebra. It's obvious that $(\frakL, \frakm,\iota_\frakm, \fraka, \iota_\fraka, P,\Delta)$ is a sub-generalised V datum of $(\frakL, \frakL,\iota_\frakL, \fraka, \iota_\fraka, P,\Delta) $ and  $s \frak{m}\oplus \fraka$ is a sub-$L_\infty[1]$-algebra of $s \frakL\oplus \fraka$.
      \end{rmk}

      Obviously we have the modified version of Theorem~\ref{theo:V-data}.

      \begin{prop} \label{L infty}
      	Let $(\frakL, \frakm,\iota_\frakm, \fraka, \iota_\fraka, P,\Delta)$ be a generalised V-datum. If $P\circ \iota_\frakm=0$ and $\Delta =0$, then  there exists  an $L_\infty[1]$-algebra  structure on $s \frakm\oplus{\fraka}$, where the higher Lie bracket $\{l_i\}_{i=1}^{+\infty}$ are given by
      	$$\begin{array}{rcl}
      		l_2(sf,sg)   &  =    &(-1)^{|f|} s[f,g]_\frakm , \\
      		l_i(sf,\xi_1,\cdots,\xi_{i-1})   &  =    & P[\cdots[\iota_\frakm(f),\iota_{{\fraka}}(\xi_1)]_{\frakL},\cdots,\iota_{{\fraka}}(\xi_{i-1})]_{\frakL} ,  i\geq2,\\
      	\end{array}$$
      	for homogeneous elements $f,g\in\frakm$, $\xi_1,\cdots,\xi_i\in{\fraka}$ and all other components vanish.
      \end{prop}

      \begin{prop} \label{lambda L infty}
      	Use the notation given in the proposition above and  $\mu\in\bfk$, then  there exists  a modified $L_\infty[1]$-algebra  structure on $s \frakm\oplus{\fraka}$, where the higher Lie bracket $\{l_i\}_{i=1}^{+\infty}$ are given by
      	$$\begin{array}{rcl}
      		l'_2(sf,sg)   &  =    &(-1)^{|f|} s[f,g]_\frakm , \\
      		l'_i(sf,\xi_1,\cdots,\xi_{i-1})   &  =    &\mu^{|f|+2-i} P[\cdots[\iota_\frakm(f),\iota_{{\fraka}}(\xi_1)]_{\frakL},\cdots,\iota_{{\fraka}}(\xi_{i-1})]_{\frakL} ,  i\geq2,\\
      	\end{array}$$
      	for homogeneous elements $f,g\in\frakm$, $\xi_1,\cdots,\xi_i\in{\fraka}$ and all other components vanish.
      \end{prop}
      \begin{proof}
      	
      	For general $n\geq 3$,  we just consider $x_1, x_2\in\frakm$ and others belong to $\fraka .$
      	We have $$\sum_{i=1}^n\sum_{\sigma\in \Sh(i,n-i)}\varepsilon(\sigma)\mu^{|x_1|+|x_2|+3-n} l_{n-i+1}(l_i(x_{\sigma(1)},\dots,x_{\sigma(i)}),x_{\sigma(i+1)},\dots,x_{\sigma(n)})=0.$$\\
      	Then generalised Jacobi identity hold for this modified $L_\infty[1]$-struceture.
      	
      \end{proof}

      \medskip

      \subsection{The $L_{\infty}$-structure for Rota-Baxter Lie algebras with weight }\

      Let $ \frak g$ be a vector space. Then we have an ordinary graded Lie algebra   $$\frakL:=\mathrm{Hom}(\wedge(\frakg\oplus\mathfrak{g}),\frakg\oplus\mathfrak{g})$$ equipped with the Nijenhuis-Richardson bracket $[-,-]_{\mathrm{NR}}$.  For convenience of presentation, we shall write $\frakg'$ for the second $\frakg$ in $\frakg\oplus \frakg$,
      denoted by $$\frakL:=\mathrm{Hom}(\wedge(\frakg\oplus\mathfrak{g}'),\frakg\oplus\mathfrak{g}').$$
      Let $$\frakm:=\mathrm{Hom}(\wedge\frakg,\frakg)\ \mathrm{and}\
      \fraka:=\Hom(\wedge\frakg,\frakg).$$
      
      Define $\iota_\frakm: \frakm\to \frakL$ by:
      	for given $f:\wedge^{n+1}\frakg\to\frakg \in \frakm$,
      $\iota_m(f)=f_0+\sum\limits_{i=1}^{n}f_i+f_{n+1}$
      where   $f_0:\wedge^{n+1}\frakg\to\frakg$, $f_i:\wedge^i\frakg\otimes\wedge^{n+1-i}\frakg'\to\frakg'$  and $f_{n+1}:\wedge^{n+1}\frakg'\to\frakg'$ is given by
      $$f_i(x_1,\dots , x_{n+1}):=f(x_1,\dots , x_{n+1}), \text{ for } 0\le i\le n+1.$$

      $\iota_\fraka$ identifies $\fraka =\Hom(\wedge\frakg,\frakg) $ with the subspace $ \Hom(\wedge \frakg',\frakg) $ of
      $\frakL$.
      
      Let $P:\frakL\to \fraka$ be the natural projection identifying the subspace $ \Hom(\wedge\frakg',\frakg) $ with  $\fraka =\Hom(\wedge\frakg,\frakg) $.

      \begin{thm}\label{thm: Linfinity for wight}
      	Keeping the above notation, then we have a generalised V-datum $(\frakL,\frakm,\iota_m,\fraka,\iota_a,P,0)$. Hence there exists an $L_\infty[1]$-algebra structure on  $s \frakm\oplus\fraka$, where $l_i$ are given by
      		$$
      	l_2(sf,sg)      =    (-1)^{|f|} s[f,g]_{\NR}, $$
      	and for $i\geq 2,n=i-2$,
      	\begin{equation*}
      		\begin{aligned}
      			l_i(sf,\xi_1,\cdots,\xi_{i-1})=&\sum\limits_{\tau\in\Sh(m_1+1,\dots,m_{i-1}+1)}\\ &
      		sgn(\tau)\big(f(\xi_1\otimes\dots\otimes\xi_{i-1})\big)\tau^{-1}\\
      		&-\sum_{k=1}^{i-1}\sum\limits_{\tau\in\Sh(m_1+1,\dots,m_{i-1}+1,1,m_k)}(-1)^{(n+\sum\limits_{j=1 }^{k-1}m_j)m_k} \\ 
      		&sgn(\tau)\xi_k  (f(\xi_1\otimes\dots\otimes\xi_{i-1}\otimes\Id)\otimes\Id^{m_k})\tau^{-1}.
      		\end{aligned}
      	\end{equation*}
      	for $i\geq 2,n>i-2$,
      	\begin{equation*}
      		\begin{aligned}
      			l_i(sf,\xi_1,\cdots,\xi_{i-1})=
      			&\sum_{k=1}^{i-1}\sum\limits_{\tau\in\Sh(m_1+1,\dots,m_{i-1}+1,n+3-i,m_k)}(-1)^{(n+\sum\limits_{j=1 }^{k-1}m_j)m_k} \\ 
      			&-sgn(\tau)\mu^{n-i+2} \xi_k  \big(f(\xi_1\otimes\dots\otimes\xi_{i-1}\otimes\Id^{\otimes n+3-i})\o \Id^{\otimes m_k}\big)\tau^{-1},
      		\end{aligned}
      	\end{equation*}
      	for homogeneous elements	$f\in\Hom(\wedge^{n+1}\frakg,\frakg)\subseteq \frakm$, $g\in\Hom(\wedge^{m+1}\frakg,\frakg)\subseteq \frakm$ and $\xi_j\in\Hom(\wedge^{m_j+1}\frakg,\frakg)\subseteq \fraka$, $1\leq j\leq i-1$, and all other components vanish.
      	
      \end{thm}
     
      \begin{proof}
      	It's sufficient to show $\iota_\frakm$ is an ordinary graded Lie homomorphism.

      		In fact, for arbitrary  $f:\wedge^{n+1}\frakg\to\frakg$ and $g:\wedge^{m+1}\frakg\to\frakg$ in $\frakm$, we have
      		\begin{equation*}
      			\begin{aligned}
      				&[\iota_\frakm(f),\iota_\frakm(g)]_{\NR}\\
      				=&	\bigg[\sum_{i=0}^{n+1}f_i,\sum_{j=0}^{m+1}g_j\bigg]_{\NR}\\
      				=&[f_0,g_0]_{\NR}+(-1)^{mn+1}\sum_{j=1}^{m}g_j\bar{\circ} f_0+\sum_{i=1}^{n}f_i\bar{\circ} g_0+\sum_{i=1}^{n}\sum_{j=1}^{m}[f_i,g_j]_{\NR}
      				+\sum_{i=1}^{n}[f_i,g_{m+1}]_{\NR}\\
      				&+\sum_{j=1}^{m}[f_{n+1},g_{j}]_{\NR}+[f_{n+1},g_{m+1}]_{\NR}\\
      				=& \left(f_0\bar{\circ} g_0+\sum_{i=1}^{n}f_i\bar{\circ} g_0+\sum_{i=1}^{n}\sum_{j=1}^{m}f_i\bar{\circ}  g_j
      				+\sum_{i=1}^{n} f_i\bar{\circ} g_{m+1}+\sum_{j=1}^{m}f_{n+1}\bar{\circ} g_{j}+f_{n+1}\bar{\circ} g_{m+1}\right)-\\
      				&(-1)^{mn}\left(g_0\bar{\circ} f_0+\sum_{j=1}^{m}g_j\bar{\circ} f_0+\sum_{i=1}^{n}\sum_{j=1}^{m}g_j\bar{\circ}  f_i
      				+\sum_{j=1}^{m}g_j \bar{\circ} f_{n+1}+\sum_{i=1}^{n} g_{m+1}\bar{\circ} f_{i}+g_{m+1}\bar{\circ} f_{n+1}\right)\\
      				=&\iota(f\bar{\circ} g)-(-1)^{mn} \iota(g\bar{\circ} f)\\
      				=&\iota_\frakm([f,g]_{\NR}).
      			\end{aligned}
      		\end{equation*}
      	
      	The explicit $L_\infty[1]$-algebra structure is given by Theorem \ref{theo:V-data}, Proposition \ref{lambda L infty} and the expansion of iterated Nijenhuis-Richardson brackets.
      \end{proof}

	\end{document}